\documentclass[11pt,a4paper,reqno]{amsart}
\usepackage[a4paper,lmargin=2.5cm,rmargin=2cm,tmargin=4cm,bmargin=4cm]{geometry}

\usepackage[english]{babel}

\usepackage[centertags]{amsmath}
\usepackage{amsfonts}
\usepackage{amssymb}
\usepackage{amsthm}
\usepackage{dsfont}
\usepackage{mathtools}

\usepackage[pdftex]{color}
\usepackage[bookmarks=true,hyperindex,pdftex,colorlinks, citecolor=MidnightBlue,linkcolor=MidnightBlue,
urlcolor=MidnightBlue
]{hyperref}
\usepackage[dvipsnames]{xcolor}

\usepackage{enumerate}
\usepackage{enumitem}

\theoremstyle{plain}
\newtheorem{theorem}{Theorem}[section]
\newtheorem{proposition}[theorem]{Proposition}
\newtheorem{lemma}[theorem]{Lemma}
\newtheorem{corollary}[theorem]{Corollary}

\theoremstyle{definition}
\newtheorem{definition}[theorem]{Definition}

\newtheorem{example}[theorem]{Example}

\theoremstyle{remark}
\newtheorem{remark}[theorem]{Remark}

\newcommand{\K}{\mathbb{K}}
\newcommand{\R}{\mathbb{R}}
\newcommand{\N}{\mathbb{N}}
\newcommand{\C}{\mathbb{C}}

\DeclareMathOperator{\co}{conv}
\DeclareMathOperator{\conv}{conv}
\DeclareMathOperator{\aconv}{aconv}

\DeclareMathOperator{\SB}{SB}

\newcommand{\cconv}{\overline{\conv}}
\newcommand{\acconv}{\overline{\aconv}}

\renewcommand{\geq}{\geqslant}
\renewcommand{\leq}{\leqslant}

\newcommand{\norm}[1]{\left\Vert#1\right\Vert}

\newcommand{\pten}{\ensuremath{\widehat{\otimes}_\pi}}

\newcommand{\ext}{\operatorname{ext}}

\begin{document}

\title[Weak diametral diameter two property]{Banach spaces with the \\ weak diametral diameter two property}

\author[Guerrero-Viu]{Juan Guerrero-Viu}
\address[Guerrero-Viu]{Departamento de Matemáticas, Universidad de Zaragoza, 50009, Zaragoza, Spain} 
\email{j.guerrero@unizar.es}
\urladdr{\href{https://orcid.org/0009-0001-2125-5120}{ORCID: \texttt{0009-0001-2125-5120}}}

\author[Mart\'in]{Miguel Mart\'in}
\address[Mart\'in]{Department of Mathematical Analysis and Institute of Mathematics (IMAG), University of Granada, E-18071 Granada, Spain}
\email{mmartins@ugr.es}
\urladdr{\url{https://www.ugr.es/local/mmartins/}}
\urladdr{
\href{http://orcid.org/0000-0003-4502-798X}{ORCID: \texttt{0000-0003-4502-798X} } }

\author[Rueda Zoca]{Abraham Rueda Zoca}
\address[Rueda Zoca]{Department of Mathematical Analysis and Institute of Mathematics (IMAG), University of Granada, E-18071 Granada, Spain}
\email{abrahamrueda@ugr.es}
\urladdr{\url{https://arzenglish.wordpress.com}}
\urladdr{
\href{https://orcid.org/0000-0003-0718-1353}{ORCID: \texttt{0000-0003-0718-1353} } }



\begin{abstract}
We introduce and systematically study the weak diametral diameter two property (weak-DD2P), a new geometric property that lies strictly between the diametral diameter two property and both the diameter two property and the convex diametral local diameter two property. A necessary condition for the weak-DD2P is obtained through the structure of the extreme points of the dual unit ball, which leads to characterisations for several classical classes of spaces, including $C(K)$ spaces, $L_1$-preduals, unital uniform algebras, and (vector-valued) function algebras. We establish stability results under standard constructions such as absolute sums, Köthe--Bochner spaces, and projective (symmetric) tensor products. Moreover, we provide complete descriptions of the weak-DD2P for vector-valued spaces of the form $L_1(\mu,X)$, $L_\infty(\mu,X)$, and $C(K,X)$. These results yield a wide range of new examples and show that the weak-DD2P exhibits a behaviour genuinely different from that of other diameter two properties.
\end{abstract}
\maketitle

\begin{center}
\begin{minipage}{.8\textwidth}
  \centering
\parskip=0ex
    \tableofcontents
\end{minipage}
\end{center}

\section{Introduction}

The study of the geometry of slices, relatively weakly open sets, and convex combinations of slices has played a central role in several developments in Functional Analysis. On one end of the spectrum, the existence of slices (respectively, non-empty relatively weakly open subsets, convex combinations of slices) of arbitrarily small diameter or arbitrarily small radius in every closed, convex and bounded subset of a Banach space yields characterisations of fundamental geometric properties, namely the Radon--Nikodým property (RNP), the convex point of continuity property (CPCP), and strong regularity (SR), respectively.

At the opposite extreme, the study of \emph{large} slices, relatively weakly open subsets, and convex combinations of slices has been carried out through the so-called diameter two properties. Recall that a Banach space $X$ is said to have the \emph{strong diameter two property} (SD2P), respectively the \emph{diameter two property} (D2P) or the \emph{local diameter two property} (LD2P), if every convex combination of slices, respectively every non-empty relatively weakly open subset, or every slice of the unit ball of $X$ has diameter two. We refer the reader to \cite{ahntt,almt21,aln13,blr15eje1,blr15eje2} and the references therein for background on these properties. Stronger notions, incorporating explicit \emph{diametrality} conditions, were later introduced in \cite{blr18} as follows.

\begin{definition}\label{defiDD2P}
Let $X$ be a Banach space. We say that $X$ has the:
\begin{enumerate}
    \item the \textit{diametral local diameter two property} (\textit{DLD2P}) if given any slice $S$ of
$B_X$, $x\in S\cap S_X$ and $\varepsilon\in\mathbb R^+$, there
exists $y\in S$ such that
$\Vert x-y\Vert>2-\varepsilon.$
    \item the \textit{diametral diameter two property}
(\textit{DD2P}) if given any non-empty relatively weakly open subset $W$ of
$B_X$, $x\in W\cap S_X$ and $\varepsilon\in\mathbb R^+$, there
exists $y\in W$ such that
$\Vert x-y\Vert>2-\varepsilon$.
\end{enumerate}
\end{definition}

Observe that a diametral version for the strong diameter two property was also defined in \cite{blr18}, but V.\ Kadets proved that this property is equivalent to the classical Daugavet property \cite[Theorem 2.3]{kadets21}.

It is clear that each of the above properties implies its corresponding non-diametral version. Moreover, while it is known that the LD2P does not imply the D2P \cite{blr15eje1}, the question of whether the DLD2P implies the DD2P (or even the D2P) remains open \cite[Question~4.1]{blr18}.

More recently, localised versions of the above properties have been studied in depth, together with the question of how these local notions influence the topology and geometry of a given Banach space. To this end, we introduce some standard notation.

\begin{definition}
Let $X$ be a Banach space and let $x\in S_X$. We say that $x$ is
\begin{enumerate}
\item[\cite{ahlp}] a \emph{$\Delta$-point} if, for every slice $S$ of $B_X$ with $x\in S$ and every $\varepsilon>0$, there exists $y\in S$ such that $\|x-y\|>2-\varepsilon$;
\item[\cite{MPR}] a \emph{super $\Delta$-point} if, for every non-empty relatively weakly open subset $W$ of $B_X$ with $x\in W$ and every $\varepsilon>0$, there exists $y\in W$ such that $\|x-y\|>2-\varepsilon$;
\item[\cite{MPR}] a \emph{ccs $\Delta$-point} if, for every convex combination of slices $C$ of $B_X$ with $x\in C$ and every $\varepsilon>0$, there exists $y\in C$ such that $\|x-y\|>2-\varepsilon$.
\end{enumerate}
\end{definition}

The notions introduced above can be viewed as localisations of the diametral diameter two properties. Indeed, a Banach space $X$ has the DLD2P (respectively, the DD2P) if and only if every point of $S_X$ is a $\Delta$-point (respectively, a super $\Delta$-point).

A substantial body of research has emerged on $\Delta$-points and super $\Delta$-points since their introduction, covering a wide range of directions: their study in specific classes of Banach spaces \cite{lp,lrt,VeeorgStudia}, the analysis of the geometric consequences of their presence in a Banach space \cite{aalmmppv,almp,MPR}, and the investigation of the minimal number of $\Delta$-points or super $\Delta$-points required to ensure diameter two properties \cite{ahlp,almt21,taller}. 

Besides, the question of how many $\Delta$-points a Banach space needs in order to have the LD2P has attracted considerable attention. Observe that the mere presence of $\Delta$-points is not sufficient in order to guarantee that a Banach space $X$ has the LD2P as shows the fact that there exists an equivalent renorming of $\ell_1$ with a $\Delta$-point \cite[Theorem 4.1]{aalmmppv}, while the space satisfies the RNP. However, if we require that the set of $\Delta$-points generates the unit ball by taking closed convex hull, then we clearly have the LD2P \cite{ahlp}, as this property is equivalent to require that every slice contains a $\Delta$-point. This motivates the following definition from \cite{ahlp}.

\begin{definition}
Let $X$ be a Banach space. We say that $X$ has the \textit{convex-DLD2P} if for every slice $S$ of $B_X$ there exists a $\Delta$-point $x\in S$, equivalently, if $B_X$ is the closed convex hull of the set of $\Delta$-points.
\end{definition}

It is obvious that the DLD2P implies the convex-DLD2P, which in turn clearly implies the LD2P, and it is known that no implication reverses \cite{ahlp}. We refer the reader to \cite{ahlp,MPR,MartinRueda-symm} for background and results on the convex-DLD2P.

The present work is concerned with the following question: given a Banach space $X$, how large does the set of super $\Delta$-points need to be for $X$ to have the D2P?

Observe that requiring the set of super $\Delta$-points to generate the unit ball by closed convex hull is not enough. Indeed, in \cite{taller} it is proved that there exists a Banach space $X$ isomorphic to $L_\infty([0,1])$ where the set of super $\Delta$-points of $S_X$ generates the unit ball by closed convex hull and, at the same time, there are non-empty relatively weakly open subsets of arbitrarily small diameter \cite[Theorem 1.3]{taller}. This points out that the next step could be the possibility that the set of super $\Delta$-points be weakly dense in the unit ball. This motivates us to consider the following definition.

\begin{definition}
Let $X$ be a Banach space. We say that $X$ has the \textit{weak-DD2P} if given any non-empty weak open subset $W$ of $B_X$ there exists $x\in W$ such that 
$$\sup\nolimits_{y\in V}\Vert x-y\Vert=2$$ for every weak open set $V\subseteq B_X$ with $x\in V$.
\end{definition}

It is immediate that a Banach space $X$ has the weak-DD2P if and only if the set of super $\Delta$-points is weakly dense in $B_X$. As a consequence, the weak-DD2P implies the convex-DLD2P. Furthermore, the DD2P implies the weak-DD2P, and the latter implies the D2P. We show in Example~\ref{example:distintas} that none of these implications can be reversed.

The present work contributes to the ongoing programme aimed at understanding the fine geometric structure of Banach spaces through the behaviour of slices and relatively weakly open subsets of the unit ball. While diameter two phenomena have been extensively studied in recent years, both in their global and local forms, several natural questions remain open concerning the precise hierarchy and mutual independence of the corresponding properties. In this context, the weak diametral diameter two property captures, as we will see, a genuinely intermediate geometric behaviour. Our results reveal that this property cannot be reduced to any of the previously known diameter two conditions and, at the same time, exhibits a remarkable stability under a wide range of classical constructions. This interplay between geometric rigidity and flexibility leads, as shown throughout the paper, to unexpected examples and sharp characterisations in fundamental classes of Banach spaces, highlighting new structural features that were not detected by earlier approaches.

The paper is organised as follows. Section~\ref{section:def} is devoted to the study of the weak-DD2P in several families of Banach spaces. In Proposition~\ref{prop:denseC(K)} we characterise the spaces $C(K)$ with the weak-DD2P as those for which $K$ has infinitely many cluster points. This result suffices to distinguish between the DD2P, the weak-DD2P and the convex-LDD2P, and it leads to a necessary condition for the weak-DD2P formulated in terms of the extreme points of the dual unit ball (Proposition~\ref{prop:condineceweakdd2p}). Next, we establish sufficient conditions for the weak-DD2P for $L_1$-preduals and for unital uniform algebras (Corollaries~\ref{cor:charweakdd2Pl1predual} and~\ref{cor:charweakdd2p-unitaluniformalgebras}), as particular cases of a general result for spaces with a norming $\ell_1$-structure (Theorem~\ref{thm:sufficient-for-ell1-norming}).
Finally, the last subsection is devoted to the slightly more general framework of vector-valued function algebras. In Proposition~\ref{prop:superDeltasubalgebraX} we exhibit a family of super $\Delta$-points in such spaces, and in Theorem~\ref{thm:DenseSubalgebrasX} we show that a vector-valued function algebra has the weak-DD2P provided that $\SB(A)'\cap \SB(A)$ is infinite, where $\SB(A)$ denotes the set of strong boundary points of the base algebra $A$.

Next, Section~\ref{section:stability} is devoted to the study of several stability properties of the weak-DD2P under standard operations on Banach spaces. In Subsection~\ref{subsection:absolutesums} we provide a complete description of the behaviour of the weak-DD2P with respect to absolute sums of Banach spaces, showing that it parallels that of the D2P. More precisely, for the $\ell_\infty$-sum of two spaces, $X\oplus_\infty Y$ has the weak-DD2P if and only if one of the spaces $X$ or $Y$ has the weak-DD2P (Propositions~\ref{prop:densesuminfty} and~\ref{prop:sumasabsolutasbaja}). For the remaining absolute sums, $X\oplus_N Y$ has the weak-DD2P if and only if both $X$ and $Y$ have the weak-DD2P (Propositions~\ref{prop:sumabsosuben} and~\ref{prop:sumasabsolutasbaja}). In Subsection~\ref{subsection:kothe} we show that the weak-DD2P passes from a Banach space $X$ to the vector-valued Köthe function space $E(X)$, provided that the underlying measure space is complete and $\sigma$-finite. Finally, Subsection~\ref{subsection:tensores} contains an in-depth analysis of stability results for the weak-DD2P under projective tensor products and their symmetric counterparts. In Theorem~\ref{theo:maintheotensornosim} we prove that $A(\Omega,X)\pten Y$ has the weak-DD2P whenever the function algebra $A(\Omega,X)$ satisfies the hypotheses of Proposition~\ref{prop:superDeltasubalgebraX} and the space $Y$ satisfies that $Y^*$ contains a spear vector or that every convex combination of slices of $B_Y$ intersects the unit sphere. Moreover, for the projective symmetric tensor product we show in Theorem~\ref{theo:proyectivosimetrico} that $\widehat{\otimes}_{\pi,s,N} C(K)$ has the weak-DD2P whenever $C(K)$ has the weak-DD2P.

Finally, in Section~\ref{section:espaclavectorval} we exploit the results obtained in Subsection~\ref{subsection:absolutesums} to provide a complete description of the weak-DD2P for the vector-valued spaces $L_1(\mu,X)$, $L_\infty(\mu,X)$ and $C(K,X)$. 

In the case of $L_1(\mu,X)$, we show that this space has the weak-DD2P if and only if the measure $\mu$ is atomless or the space $X$ has the weak-DD2P. For $L_\infty(\mu,X)$, we prove that the weak-DD2P is equivalent to any of the following conditions: $\mu$ is not purely atomic, $\mu$ has infinitely many atoms, or $X$ has the weak-DD2P. 

For spaces of the form $C(K,X)$ we obtain the following characterisation:
\begin{itemize}[nolistsep]
    \item[-] if $K'=\emptyset$, then $C(K,X)$ has the weak-DD2P if and only if $X$ has the weak-DD2P;
    \item[-] if $K'$ is a non-empty finite set, then $C(K,X)$ has the weak-DD2P if and only if $X$ is infinite-dimensional;
    \item[-] if $K'$ is infinite, then $C(K,X)$ has the weak-DD2P regardless of the space $X$.
\end{itemize}
As a noteworthy consequence of this description, we show that a space of the form $C(K,X)$ may have the weak-DD2P even when both $C(K)$ and $X$ fail to satisfy the property, as illustrated by the example $c(\ell_2)$ (see Remark~\ref{rem:C(K,X)vecval}).

\textbf{Terminology.}
Let $X$ be a Banach space over the field $\K$, where $\K=\R$ or $\K=\C$. We denote by $B_X$ and $S_X$ the closed unit ball and the unit sphere of $X$, respectively. The symbol $X^*$ stands for the (topological) dual space of $X$, and $J_X\colon X\to X^{**}$ denotes the canonical embedding of $X$ into its bidual. Given a subset $C\subset X$, we write $\co(C)$ for the convex hull of $C$ and $\aconv(C)$ for its absolutely convex hull.

\section{The first examples}\label{section:def}
We study the weak-DD2P in several families of Banach spaces: $C(K)$ spaces, spaces with norming $\ell_1$-structure (which include preduals of $L_1$ and unital uniform algebras), and (vector valued) function algebras.

\subsection{Spaces of continous functions} 
As noted in the introduction, the weak-DD2P is distinct from the DD2P, the convex-LD2P, and the D2P. In order to make this distinction explicit, we present a characterisation of the weak-DD2P for $C(K)$ spaces. While more general results will follow in the next subsections, this characterisation is sufficient for our purposes and provides motivation for the results that come next.

Given a topological space $V$, we denote by $V'$ the set of cluster points of $V$. For a subset $A\subseteq V$, we write $A'$ for the set of cluster points of $A$ in $V$. 

\begin{proposition}\label{prop:denseC(K)}
    Let $K$ be a compact Hausdorff topological space. Then $C(K)$ has the weak-DD2P if and only if $K'$ is an infinite set (equivalently, if $K''\neq \emptyset$).
\end{proposition}

\begin{proof}
   If $K'=\{t_1,\ldots,t_n\}$, define
\[
W=\left\{ f\in B_{C(K)} : |f(t_i)|<\tfrac{1}{2}, \ \forall\, i=1,\ldots,n \right\},
\]
which is clearly a weakly open subset of $B_{C(K)}$. Now, if $f$ is any super $\Delta$-point, there exists some $i\in\{1,\ldots,n\}$ such that $|f(t_i)|=1$ (see \cite[Proposition~4.1]{MPR}), and hence $f\notin W$. Consequently, super $\Delta$-points are not weakly dense in $B_{C(K)}$.

Conversely, assume that $K'$ is infinite and let $f\in B_{C(K)}$. We shall construct a sequence of super $\Delta$-points in $S_{C(K)}$ that converges weakly to $f$. To this end, choose a sequence $(t_n)_{n}\subseteq K'$ such that $t_n\neq t_m$ whenever $n\neq m$. Then there exists a sequence $(U_n)_n$ of pairwise disjoint open subsets of $K$ satisfying $t_n\in U_n$ for all $n\in\mathbb{N}$. For each $n\in\mathbb{N}$, let $g_n\colon K\to[0,1]$ be a Urysohn function such that $g_n(t_n)=1$ and $g_n(t)=0$ for all $t\notin U_n$. Define a sequence $(f_n)_n\subseteq B_{C(K)}$ by
\[
f_n := (1-g_n)f + g_n.
\]
First, observe that
\[
\|f_n\|\leq \sup_{t\in K}\big( (1-g_n(t))|f(t)| + g_n(t) \big)\leq 1,
\]
since $\|f\|\leq 1$ and $g_n(t)\in[0,1]$ for all $t\in K$. Moreover, $f_n(t_n)=1$ and $f_n(t)=f(t)$ for all $t\notin U_n$. By \cite[Theorem~4.2]{MPR}, it follows that each $f_n$ is a super $\Delta$-point. Finally, the pairwise disjointness of the sets $U_n$ implies that $(f_n)_n$ converges pointwise to $f$. Therefore, by Rainwater’s theorem (see, for instance, \cite[Theorem~3.134]{RojoYAmarillo}), the sequence $(f_n)_n$ converges weakly to $f$.
\end{proof}

Before continuing with the results of the section, we include the following remark, which is of independent interest.

\begin{remark}
The argument in the proof above actually yields a stronger conclusion than the weak-DD2P itself. Indeed, since the sequence $(g_n)_n$ constructed there does not depend on the particular choice of the function $f$, the argument can be pushed further. More precisely, let $K$ be a compact Hausdorff space such that $K'$ is infinite. Given functions $f_1,\ldots,f_n\in S_{C(K)}$ and scalars $\lambda_1,\ldots,\lambda_n\geq 0$ with $\sum_{i=1}^n \lambda_i=1$, one can construct sequences $(f_1^k)_k,\ldots,(f_n^k)_k$ in $S_{C(K)}$ such that $f_i^k\rightharpoonup f_i$ for every $1\leq i\leq n$ and, moreover, for each $k\in\mathbb{N}$ the function
\[
\sum_{i=1}^n \lambda_i f_i^k
\]
is simultaneously a ccs $\Delta$-point and a super $\Delta$-point. This follows from the fact that, in $C(K)$ spaces, both notions coincide (see \cite[Corollary~4.3]{MPR}). In fact, in this setting a function $f\in S_{C(K)}$ is a super $\Delta$-point (and hence a ccs $\Delta$-point) if and only if it attains its norm at some cluster point of $K$. As a consequence, it can be shown that for every convex combination of slices $C$ of $B_{C(K)}$ there exists a function $f\in C$ which is a ccs $\Delta$-point.

Concerning the Banach space property that \emph{every convex combination of slices contains a ccs $\Delta$-point}, the following observations are in order:
\begin{enumerate}
    \item This property implies that every convex combination of slices has diameter two, since it ensures that every such combination intersects the unit sphere; hence \cite[Theorem~3.4]{lmr} applies.
    
    \item The above property does not imply the Daugavet property. Indeed, if $K=[0,\omega_2]$, where $\omega_2$ denotes the second uncountable ordinal, then $C(K)$ fails the Daugavet property because $K$ contains isolated points. Nevertheless, $K$ has infinitely many cluster points, and thus $C(K)$ satisfies the property described above.
    
    \item The weak-DD2P does not imply the above property. Indeed, if $X$ is a Banach space with the weak-DD2P, Proposition~\ref{prop:sumabsosuben} shows that $X\oplus_2 X$ has the weak-DD2P. However, the $\ell_2$-sum of any two Banach spaces fails the SD2P (see \cite[Theorem~3.2]{abl}).
\end{enumerate}

Altogether, these observations show that the property that every convex combination of slices contains a ccs $\Delta$-point is not equivalent to any of the previously known diameter two properties.
\end{remark}

Next, with the aid of Proposition~\ref{prop:denseC(K)}, we may distinguish the weak-DD2P from several previously studied properties.

\begin{example}\label{example:distintas}
The following examples illustrate that the weak-DD2P differs from a number of existing properties.
\begin{enumerate}
    \item It is well known that $c_0$ has the D2P, but fails the weak-DD2P and even the convex-DLD2P, since it does not contain any $\Delta$-point (see \cite[Example~2.4.5]{Pirkthesis} for further details).
    
    \item The space $c=C([0,\omega_0])$ enjoys the convex-DLD2P, as the set of super $\Delta$-points generates its unit ball as a closed convex hull (see \cite[Proposition~2.4.3]{Pirkthesis}). However, it fails the weak-DD2P by Proposition~\ref{prop:denseC(K)}.
    
    \item Let $K$ be a compact Hausdorff topological space with at least one isolated point and such that $K'$ is infinite (for instance, $K=[0,1]\cup\{2\}$). Then $C(K)$ has the weak-DD2P by Proposition~\ref{prop:denseC(K)}. Nevertheless, it fails the DD2P, since it is easy to construct norm-one functions that do not attain their norm at a cluster point. By \cite[Theorem~4.2]{MPR}, such functions cannot be super $\Delta$-points.
    
    \item Let $X$ be a Müntz space $M_0(\Lambda)$ with $\lambda_1\geq 1$. Then $X$ has the convex-DLD2P by virtue of \cite[Theorem~5.7]{ahlp}. We claim, however, that $X$ fails the weak-DD2P. Indeed, let $f\in B_X$ satisfy $|f(1)|<1$ (for example, $f(t)=t^{\lambda_1}-t^{\lambda_2}$). Suppose, towards a contradiction, that there exists a net $(f_\alpha)_\alpha$ of super $\Delta$-points converging weakly to $f$. By \cite[Theorem~3.13]{ahlp}, we have $|f_\alpha(1)|=1$ for all $\alpha$. Consequently, the net $(\delta_1(f_\alpha))_\alpha$ cannot converge to $\delta_1(f)$, which yields the desired contradiction. We refer the interested reader to \cite{ahlp,Pirkthesis} for more information on diameter two properties on Müntz spaces.
\end{enumerate}
\end{example}

Our next goal is to extend the sufficient condition from Proposition~\ref{prop:denseC(K)} to the space $C_0(L,X)$ of vector-valued continuous functions on a locally compact Hausdorff space $L$ which vanish at infinity. This can be achieved with essentially the same argument.

\begin{proposition}\label{prop:C0(L,X)}
Let $L$ be a locally compact Hausdorff space and let $X$ be a Banach space. If $L'$ is infinite, then $C_0(L,X)$ has the weak-DD2P.
\end{proposition}

\begin{proof}
The proof is a straightforward adaptation of that of Proposition~\ref{prop:denseC(K)}. Fix $f\in B_{C_0(L,X)}$ and choose a sequence $(t_n)_n\subseteq L'$ with $t_n\neq t_m$ whenever $n\neq m$. Let $(U_n)_n$ be a sequence of pairwise disjoint open subsets of $L$ such that $t_n\in U_n$ for all $n\in\mathbb{N}$. For each $n\in\mathbb{N}$, let $g_n\colon L\to[0,1]$ be a Urysohn function satisfying $g_n(t_n)=1$ and $g_n(t)=0$ for all $t\notin U_n$. Fix $x_0\in S_X$ and define
\[
f_n \coloneqq (1-g_n)f + g_n x_0, \qquad n\in\mathbb{N}.
\]
It is immediate that $(f_n)_n\subseteq B_{C_0(L,X)}$. Moreover, each function $f_n$ attains its norm at the point $t_n\in L'$, and hence $f_n$ is a super $\Delta$-point by \cite[Theorem~4.1]{lrt}. 

Finally, to show that $(f_n)_n$ converges weakly to $f$, note that the pairwise disjointness of the sets $U_n$ implies that the sequence $(g_n)_n$ is equivalent to the canonical basis of $c_0$ and, therefore, it is weakly null. Since $f_n-f = g_n(x_0-f)$, it follows that $f_n\rightharpoonup f$. Since $f$ is  arbitrary, the result follows.
\end{proof}

We finish this subsection providing with a necessary condition for the weak-DD2P which is implicitly contained in Proposition~\ref{prop:denseC(K)} through the identification of the extreme points of the dual unit ball of a $C(K)$ space. Indeed, if $K$ is a compact Hausdorff space, it is well known that $K$ is homeomorphic to the weak$^*$-compact set $(\{\delta_t \colon t\in K\},w^*)$ and that
\[
\ext(B_{C(K)^*})=\{\delta_t \colon t\in K,\,\lambda\in \K,\,|\lambda|=1\}
\]
(that is, the extreme points of $B_{C(K)^*}$ are precisely the unimodular multiples of the evaluation functionals). Consequently, Proposition~\ref{prop:denseC(K)} can be rephrased as follows: the space $C(K)$ has the weak-DD2P if and only if the quotient space obtained from $\ext(B_{C(K)^*})$ by identifying linearly dependent functionals contains infinitely many isolated points. To generalise and formalise this idea, we introduce the following notation. Given a Banach space $X$, we define an equivalence relation on $S_{X^*}$ by declaring $f\sim g$ if and only if $f$ and $g$ are linearly dependent. 
We endow the quotient set $S_{X^*}/\!\sim$ with the topology induced by the weak$^*$ topology and denote by
\[
\pi\colon S_{X^*}\longrightarrow S_{X^*}/\!\sim
\]
the corresponding quotient map, which is onto, continuous and open.

We are now in a position to establish the announced necessary condition.

\begin{proposition}\label{prop:condineceweakdd2p}
    Let $X$ be a Banach space with the weak-DD2P. Then $\pi(\ext(B_{X^*}))'$ is an infinite set.
\end{proposition}

\begin{proof}
   Let $I$ be an index set for the elements of $\pi(\ext(B_{X^*}))'$, and for each $i\in I$ choose a representative $x_i^*\in S_{X^*}$ of the corresponding equivalence class in $\pi(\ext(B_{X^*}))'$. By  \cite[Remark~3.5]{MartinRueda-symm}, the set
\[
W:=\bigcap\nolimits_{i\in I}\bigl\{ x\in B_X \colon |x_i^*(x)|<\tfrac{1}{2} \bigr\}
\]
contains no super $\Delta$-points (in fact, no $\Delta$-points at all). Hence, since $X$ has the weak-DD2P, the set $A$ cannot be weakly open in $B_X$, and therefore the index set $I$ must be infinite.
\end{proof}

\begin{remark}
Let us denote by $\operatorname{NA}(X)\subseteq X^*$ the set of norm-attaining functionals, that is, those $x^*\in X^*$ for which there exists $x\in S_X$ satisfying $x^*(x)=\|x^*\|$. Using again \cite[Remark~3.5]{MartinRueda-symm}, one obtains the following slightly stronger version of Proposition~\ref{prop:condineceweakdd2p}: if a Banach space $X$ has the weak-DD2P, then the set
\[
\pi(\ext(B_{X^*}))'\cap \pi\bigl(\operatorname{NA}(X)\bigr)
\]
must be infinite.
\end{remark}

Let us observe that  Proposition~\ref{prop:condineceweakdd2p} is not sufficient in general. (Indeed, this can already be seen by considering the case $X=\ell_2$.)

\subsection{Spaces with norming \texorpdfstring{$\ell_1$}{l1}-structure}
Our next objective is to obtain a converse to Proposition~\ref{prop:condineceweakdd2p} for the class of $L_1$-preduals and for the class of unital uniform algebras. To this end, we shall work in the broader setting of Banach spaces with norming $\ell_1$-structures.

We recall that a Banach space $X$ has a \textit{norming $\ell_1$-structure} if there exists a subset $A$ of $S_{X^*}$ which is one-norming for $X$ and such that $\K a^*$ is an $L$-summand of $X^*$ for every $a^*\in A$ (i.e. $X^*=\K a^*\oplus_1 Z$ for some $Z\subseteq X^*$). We refer to \cite[Section 3.5.2]{DaugavetBook} for further information regarding this definition. Recall that a Banach space $X$ is said to have a \emph{norming $\ell_1$-structure} if there exists a subset $A\subseteq S_{X^*}$ which is one-norming for $X$ and such that $\K a^*$ is an $L$-summand of $X^*$ for every $a^*\in A$, that is,
\[
X^*=\K a^*\oplus_1 Z
\]
for some subspace $Z\subseteq X^*$. We refer the reader to \cite[Section~3.5.2]{DaugavetBook} for further background on this notion.

\begin{theorem}\label{thm:sufficient-for-ell1-norming}
    Let $X$ be a Banach space with norming $\ell_1$-structure witnessed by a set $A\subseteq S_{X^*}$. If $\pi(A)'\cap \pi(A)$ is an infinite set, then $X$ has the weak-DD2P.
\end{theorem}

Our proof of the theorem is based on the one of \cite[Theorem~3.5.14]{DaugavetBook}. We shall use the following notation. Given $x^*\in S_{X^*}$, we write
\begin{align*}
    F(B_X,x^*) &:= \{ x\in B_X : x^*(x)=1 \}.
\end{align*}

We shall need the following result, which we prefer to extract from the main proof for the sake of clarity.

\begin{lemma}\label{lemma:ell1-norming-suff-superDelta}
Let $X$ be a Banach space with a norming $\ell_1$-structure witnessed by a set $A\subseteq S_{X^*}$, and let $x^*\in A$ be such that $\pi(x^*)\in \pi(A)'$. Then every $x\in F(B_X,x^*)$ is a super $\Delta$-point.
\end{lemma}

\begin{proof}
It suffices to observe that condition~(6) in \cite[Theorem~3.5.14]{DaugavetBook} implies the definition of a super $\Delta$-point.
\end{proof}

\begin{proof}[Proof of Theorem~\ref{thm:sufficient-for-ell1-norming}]
Let $x\in B_X$. Choose a sequence $(\pi(a_n^*))_n\subseteq \pi(A)'\cap \pi(A)$ such that $\pi(a_n^*)\neq \pi(a_m^*)$ for all $n\neq m$ (that is, the functionals $a_n^*,a_m^*\in A$ are pairwise linearly independent). For each $n\in\mathbb{N}$, let $P_n$ be the $L$-projection satisfying
\[
X^*=\K a_n^*\oplus_1 \ker P_n
\]
(see \cite[Remark~3.5.9(a)]{DaugavetBook}), and define a linear functional $x_n^{**}\colon X^{**}\to\K$ by
\[
x_n^{**}(\lambda a_n^*+z^*)=\lambda+z^*(x), \qquad 
\lambda a_n^*+z^*\in \K a_n^*\oplus_1 \ker P_n=X^*.
\]
It is straightforward to verify that $x_n^{**}\in X^{**}$ with $\|x_n^{**}\|\leq 1$ for all $n\in\mathbb{N}$.

We claim that the sequence $(x_n^{**})_n$ converges weakly to $x$. To see this, define a linear operator $T\colon c_0\to X^{**}$ by
\[
T(e_n)=x_n^{**}-x, \qquad n\in\mathbb{N},
\]
where $(e_n)_n$ denotes the canonical basis of $c_0$. Fix $n\in\mathbb{N}$ and let $\lambda_1,\ldots,\lambda_n\in\K$. Set $P=P_1+\cdots+P_n$. By \cite[Lemma~2.9.4]{DaugavetBook}, $P$ is an $L$-projection, 
\[
X^*=P(X^*)\oplus_1 \ker P, \qquad 
\ker P=\bigcap_{i=1}^n \ker P_i,
\]
and
\[
B_{P(X^*)}=\aconv\{a_1^*,\ldots,a_n^*\}.
\]
Observe that, for each $i\in\{1,\ldots,n\}$ and every $z^*\in \ker P\subseteq \ker P_i$, we have
\[
x_i^{**}(z^*)-z^*(x)=0.
\]
Therefore,
\[
\Bigl\|\sum_{i=1}^n \lambda_i (x_i^{**}-x)\Bigr\|
= \max_{1\leq j\leq n}
\Bigl|\sum_{i=1}^n \lambda_i\bigl(x_i^{**}(a_j^*)-a_j^*(x)\bigr)\Bigr|.
\]
Since $x_i^{**}(a_j^*)-a_j^*(x)=0$ whenever $i\neq j$, it follows that
\[
\Bigl\|\sum_{i=1}^n \lambda_i (x_i^{**}-x)\Bigr\|
= \max_{1\leq j\leq n} |\lambda_j|\,
|x_j^{**}(a_j^*)-a_j^*(x)|
\leq 2 \max_{1\leq j\leq n} |\lambda_j|.
\]
Consequently, $T$ is bounded with $\|T\|\leq 2$, and hence weak-to-weak continuous. This proves that $(x_n^{**})_n$ converges weakly to $x$.

Now fix $n\in\mathbb{N}$. Since $\K a_n^*$ is an $L$-summand of $X^*$, we claim that
\[
x_n^{**}\in F\bigl(B_{X^{**}},J_{X^*}(a_n^*)\bigr)
= \overline{F(B_X,a_n^*)}^{\omega^*}.
\]
Indeed, we may use \cite[Definition~2.1, Example~2.12(a) and Theorem~2.9]{SpearsBook} and then the reversed Krein-Milman theorem or, alternatively,  \cite[Fact 2.9.2]{CR} and \cite[Theorem 2.3]{GI}. Hence, there exists a net $(x_\alpha)_\alpha\subseteq B_X$ converging weakly to $x$ such that, for each $\alpha$, there is some $n_\alpha\in\mathbb{N}$ with $x_\alpha\in F(B_X,a_{n_\alpha}^*)$. By Lemma~\ref{lemma:ell1-norming-suff-superDelta}, each $x_\alpha$ is a super $\Delta$-point, finishing thus the proof.
\end{proof}

As particular cases of the above theorem, we can now establish a converse of Proposition~\ref{prop:condineceweakdd2p} for the classes of $L_1$-predual spaces and unital uniform algebras. In the fist case, as it is shown in \cite[Example 3.5.11]{DaugavetBook}, every $L_1$-predual has a norming $\ell_1$-structure witnessed by the set $A=\ext(B_{X^*})$. In this setting, we denote by $E_X$ to the set $\pi(\ext(B_{X^*}))$ endowed with the quotient topology induced by the weak$^*$ topology. Observe that we are considering $E_X$ as a topological space; thus, in particular, the cluster points lie within it.

\begin{corollary}\label{cor:charweakdd2Pl1predual}
 Let $X$ be an $L_1$-predual. If $E_X'$ is an infinite set, then $X$ has the weak-DD2P.
\end{corollary}

For the second case, recall that a \emph{unital uniform algebra} is a closed subalgebra $X$ of $C(K)$, where $K$ is a compact Hausdorff topological space, which separates the points of $K$ (that is, for any $t\neq s$ in $K$ there exists $f\in X$ such that $f(t)\neq f(s)$) and contains the constant functions. The \emph{Choquet boundary} of $X$ is defined by
\[
\partial X := \{ s\in K \colon \delta_s|_X \in \ext(B_{X^*}) \},
\]
and is endowed with the topology induced from $K$. It is shown in \cite[Example~3.5.13]{DaugavetBook} that every unital uniform algebra $X$ admits a norming $\ell_1$-structure witnessed by the set
\[
A := \{ \delta_s|_X \colon s\in \partial X \}.
\]

\begin{corollary}\label{cor:charweakdd2p-unitaluniformalgebras}
Let $X$ be a unital uniform algebra. If $(\partial X)'$ is an infinite set, then $X$ has the weak-DD2P.    
\end{corollary}

\begin{proof}
It suffices to observe that, for the above set $A$, the space $\pi(A)$ is homeomorphic to $\partial X$ (use, for instance, \cite[Remark~3.5.12]{DaugavetBook}). 
\end{proof}

\subsection{Weak-DD2P for (vector-valued) function algebras} \label{subsection:functionalgebras}

In this section, we study the weak-DD2P in complex function algebras and their vector-valued counterparts. Besides providing further examples of spaces enjoying this property, this analysis will allow us to construct examples of tensor product spaces with the weak-DD2P in Subsection~\ref{subsection:tensores}. We begin by recalling the necessary definitions. In this subsection, all Banach spaces will be complex.

Let $\Omega$ be a regular Hausdorff topological space. A \emph{function algebra} on $\Omega$ is a closed subalgebra $A(\Omega)$ of the space $C_b(\Omega)$ of bounded continuous functions from $\Omega$ to $\C$ which separates the points of $\Omega$. When no confusion is likely to arise, we shall simply write $A$ instead of $A(\Omega)$. In the particular case where $\Omega$ is compact, we shall write $K=\Omega$. In this situation, a function algebra on $K$ that contains the constant functions is precisely a unital uniform algebra, as considered in the previous section.

As a vector-valued counterpart of a function algebra, as it is done in \cite{lt22}, we consider a closed subspace $A(\Omega,X)$ of $C_b(\Omega,X)$, the space of bounded continuous functions from $\Omega$ into a Banach space $X$, satisfying the following conditions:
\begin{enumerate}
    \item The base algebra
    \[
    A := \{ x^* \circ f \colon f \in A(\Omega,X),\ x^* \in X^* \}
    \]
    is a function algebra on $\Omega$.  
    \item $A\otimes X$ is contained in $A(\Omega,X)$, that is, for every $f\in A$ and every $x\in X$, the function $t\mapsto f(t)x$ belongs to $A(\Omega,X)$. 
    \item For every $f \in A$ and every $g \in A(\Omega,X)$, the pointwise product $fg$ belongs to $A(\Omega,X)$.
\end{enumerate}

In recent years, there has been a growing interest in the study of diameter two type properties and diametral points in (vector-valued) function algebras; see, for instance,  \cite{bl,cgkm,lrt,lt22}. Moreover, even though no substantial use of the deeper theory of function algebras is required here, we refer the reader to \cite{Dales, Leibowitz} for further background on this topic.

Before establishing the main results in this section we need a series of auxiliary lemmata. The first one is an adaptation of \cite[Lemma 2.2]{CGK} to our context, with exactly the same proof. 

\begin{lemma}
Let $\Omega$ be a regular Hausdorff topological space and let $A$ be a complex closed subalgebra of $C_b(\Omega)$. Let $M\subseteq \C$ and let $g\colon M\to \C$ be a function which is the uniform limit on $M$ of a sequence of complex polynomials. For every $f\in A$ such that $f(\Omega)\subseteq M$, the following statements hold:
\begin{itemize}
    \item[\upshape(a)] If $A$ is unital, then $g\circ f\in A$.
    \item[\upshape(b)] If $A$ is not unital, $0\in M$ and $g(0)=0$, then $g\circ f\in A$.
\end{itemize}
\end{lemma}

Given a subalgebra $A$ of $C_b(\Omega)$, we denote by $\SB(A)$ the set of strong boundary points of $A$. Recall that a point $\omega_0\in \Omega$ is a \textit{strong boundary point} (also called weak peak point, generalized peak point or $p$-point) if for every open set $U\subseteq \Omega$ containing $\omega_0$, there is some $f\in A$ with $f(\omega_0)=1=\norm{f}$ and $\sup_{\omega\in \Omega\setminus U}|f(\omega)|<1$.

\begin{lemma}[\mbox{\cite[Lemma 3.10]{ChoiJungTag}}]\label{lemma:engenurysohn}
Let $\Omega$ be a regular Hausdorff topological space and let $A$ be a complex closed subalgebra of $C_b(\Omega)$. If $\omega_0\in \SB(A)$, then for every open subset $U\subseteq \Omega$ containing $\omega_0$ and $\varepsilon>0$, there exists $\phi\in A$ such that $\phi(\omega_0)=1=\norm{\phi}$, $\sup_{\omega\in \Omega\setminus U}|\phi(\omega)|<\varepsilon$ and $$|\phi(\omega)|+(1-\varepsilon)|1-\phi(\omega)|\leq 1, \quad \forall \omega\in \Omega.$$
\end{lemma}

We are now ready to obtain examples of vector-valued function algebras with the weak-DD2P. To this end, we first identify, in the next proposition, a class of elements in a function algebra which are super $\Delta$-points.

\begin{proposition}\label{prop:superDeltasubalgebraX}
    Let $X$ be a complex Banach space and let $A$ be a complex function algebra on a regular Hausdorff topological space $\Omega$. Let $f\in A(\Omega,X)$ with $\|f\|=1$. If there exist $\omega_0\in \SB(A)'$ such that $\|f(\omega_0)\|=1$, then $f$ is a super $\Delta$-point.
\end{proposition}

\begin{proof}
Let $W\subseteq B_{A(\Omega,X)}$ be a weakly open set containing $f$, and let $\varepsilon>0$. Write $x_0=f(\omega_0)\in S_X$. We shall find an element $g\in W$ such that $\|f-g\|\geq 2-\varepsilon$.

For each $n\in\mathbb{N}$, define the set
\[
U_n=\left\{\omega\in\Omega \colon \|f(\omega)-x_0\|<\frac{1}{n}\right\},
\]
which is open in $\Omega$ and contains $\omega_0$. Since $\omega_0\in \SB(A)'$ and $\Omega$ is Hausdorff, there exist a sequence $(\omega_n)_n\subseteq \SB(A)$ with $\omega_n\neq \omega_m$ for $n\neq m$, and a sequence $(V_n)_n$ of pairwise disjoint open subsets of $\Omega$ such that
\[
\omega_n\in V_n\subseteq U_n \qquad \text{for all } n\in\mathbb{N}.
\]
By Lemma~\ref{lemma:engenurysohn}, for each $n\in\mathbb{N}$ there exists a function $\phi_n\in A$ such that $\phi_n(\omega_n)=1=\|\phi_n\|$,
\[
\sup_{\omega\in\Omega\setminus V_n} |\phi_n(\omega)|<2^{-n},
\]
and
\begin{equation}\label{Eq4}
|\phi_n(\omega)|+\left(1-2^{-n}\right)|1-\phi_n(\omega)|\leq 1,
\qquad \forall\,\omega\in\Omega.
\end{equation}

We claim that the sequence $(\phi_n)_n$ converges weakly to $0$. To see this, define a linear operator $T\colon c_0\to A$ by $T(e_n)=\phi_n$ for all $n\in\mathbb{N}$, where $(e_n)_n$ denotes the canonical basis of $c_0$. We show that $T$ is bounded. Let $\lambda_1,\ldots,\lambda_n\in\K$. Clearly,
\[
T\Bigl(\sum_{i=1}^n \lambda_i e_i\Bigr)=\sum_{i=1}^n \lambda_i \phi_i.
\]
If $\omega\notin \bigcup_{i=1}^n V_i$, then
\[
\Bigl|\sum_{i=1}^n \lambda_i \phi_i(\omega)\Bigr|
\leq \sum_{i=1}^n |\lambda_i|\,|\phi_i(\omega)|
\leq \sum_{i=1}^n |\lambda_i|2^{-i}
\leq \max_{1\leq i\leq n} |\lambda_i|.
\]
On the other hand, if $\omega\in \bigcup_{i=1}^n V_i$, there exists a unique index $i_0\in\{1,\ldots,n\}$ such that $\omega\in V_{i_0}$, and hence
\begin{align*}
\Bigl|\sum_{i=1}^n \lambda_i \phi_i(\omega)\Bigr|
&\leq |\lambda_{i_0}|+\sum_{\substack{i=1\\ i\neq i_0}}^n |\lambda_i|\,|\phi_i(\omega)| \\
&\leq |\lambda_{i_0}|+\sum_{\substack{i=1\\ i\neq i_0}}^n |\lambda_i|2^{-i}
\leq 2\max_{1\leq i\leq n} |\lambda_i|.
\end{align*}
Therefore,
\[
\Bigl\|\sum_{i=1}^n \lambda_i \phi_i\Bigr\|
\leq 2 \Bigl\|\sum_{i=1}^n \lambda_i e_i\Bigr\|,
\]
and hence $T$ is bounded. Consequently, $T$ is weak-to-weak continuous, and thus $(\phi_n)_n=(T(e_n))_n$ converges weakly to $0$. By an analogous argument, the sequences $(\phi_n x_0)_n$ and $(\phi_n f)_n$ in $A(\Omega,X)$ also converge weakly to $0$.

Next, define
\[
f_n:=\left(1-2^{-n}\right)(1-\phi_n)f-\phi_n x_0,
\qquad n\in\mathbb{N}.
\]
It follows from \eqref{Eq4} that $(f_n)_n\subseteq B_{A(\Omega,X)}$. Moreover, $(f_n)_n$ converges weakly to $f$. Hence, there exists $n_0\in\mathbb{N}$ such that $f_n\in W$ for all $n\geq n_0$. Choose $n_1\geq n_0$ with $\frac{1}{n_1}<\varepsilon$, and set $g=f_{n_1}\in W$. Then
\begin{align*}
\|f-g\|
&\geq \|f(\omega_{n_1})-g(\omega_{n_1})\|
= \|f(\omega_{n_1})+x_0\| \\
&\geq 2-\|f(\omega_{n_1})-x_0\|
\geq 2-\frac{1}{n_1}
\geq 2-\varepsilon,
\end{align*}
since $\omega_{n_1}\in V_{n_1}\subseteq U_{n_1}$.
\end{proof}

We now present the main theorem of this section.

\begin{theorem}\label{thm:DenseSubalgebrasX}
    Let $X$ be a complex Banach space and $A$ be a complex function algebra on a regular Hausdorff topological space $\Omega$. If the set $\SB(A)'\cap \SB(A)$ is infinite, then $A(\Omega,X)$ has the weak-DD2P.
\end{theorem}

\begin{proof}
Let $f\in B_{A(\Omega,X)}$. Choose a sequence $(\omega_n)_n\subseteq \SB(A)'\cap \SB(A)$ such that $\omega_n\neq \omega_m$ for all $n\neq m$. Since $\Omega$ is Hausdorff, there exists a sequence $(U_n)_n$ of pairwise disjoint open subsets of $\Omega$ such that $\omega_n\in U_n$ for all $n\in\mathbb{N}$. 

By Lemma~\ref{lemma:engenurysohn}, for each $n\in\mathbb{N}$ there exists a function $\phi_n\in A$ such that $\phi_n(\omega_n)=1=\|\phi_n\|$,
\[
\sup_{\omega\in \Omega\setminus U_n} |\phi_n(\omega)|<2^{-n},
\]
and
\begin{equation}\label{Eq5}
|\phi_n(\omega)|+\left(1-2^{-n}\right)|1-\phi_n(\omega)|\leq 1,
\qquad \forall\,\omega\in \Omega.
\end{equation}
As in the proof of Proposition~\ref{prop:superDeltasubalgebraX}, it follows that the sequence $(\phi_n)_n$ converges weakly to $0$. Fix $x_0\in S_X$ and define, for each $n\in\mathbb{N}$,
\[
f_n:=\left(1-2^{-n}\right)(1-\phi_n)f+\phi_n x_0.
\]
It follows from \eqref{Eq5} that $(f_n)_n\subseteq B_{A(\Omega,X)}$. Moreover, $(f_n)_n$ converges weakly to $f$. Finally, for every $n\in\mathbb{N}$ we have $f_n(\omega_n)=x_0$, and hence each $f_n$ is a super $\Delta$-point by Proposition~\ref{prop:superDeltasubalgebraX}.
\end{proof}

In the case that $A$ is a function algebra in $C_0(L)$, where $L$ is locally compact and Hausdorff, the set of strong boundary points coincide with the Choquet boundary $\partial A$ (see \cite[Theorem 2.1]{RaoRoy}). Hence, Theorem \ref{thm:DenseSubalgebrasX} yield the following. 

\begin{corollary}\label{cor:C0(L)}
Let $X$ be a complex Banach space and let $A$ be a complex function algebra in $C_0(L)$, where $L$ is a locally compact Hausdorff topological space. If the set $(\partial A)'$ is infinite, then $A(L,X)$ has the weak-DD2P.
\end{corollary}

\section{Some stability results}\label{section:stability}

In this section, we study the stability of the weak-DD2P under several standard constructions in Banach space theory. More precisely, we analyse whether this property is preserved by common operations such as absolute sums, the formation of Köthe--Bochner spaces, and tensor product constructions. These results provide further insight into the structural behaviour of the weak-DD2P and yield new examples of Banach spaces enjoying this property.

\subsection{Absolute sums}\label{subsection:absolutesums}

In this subsection, we analyse the behaviour of the weak-DD2P under absolute sums of Banach spaces. Recall that an \emph{absolute normalised norm} $N$ on $\R^2$ is a norm satisfying the following conditions:
\begin{enumerate}
    \item $N(a,b)=N(|a|,|b|)$ for all $(a,b)\in \R^2$;
    \item $N(1,0)=N(0,1)=1$.
\end{enumerate}
Given Banach spaces $X$ and $Y$, and an absolute normalised norm $N$ on $\R^2$, we denote by $X\oplus_N Y$ the product space $X\times Y$ endowed with the norm
\[
\|(x,y)\|_N := N(\|x\|,\|y\|), \qquad (x,y)\in X\times Y,
\]
and refer to $X\oplus_N Y$ as an \emph{absolute sum} of $X$ and $Y$. The most prominent examples of absolute sums are the $\ell_p$-sums, which arise when $N$ is the $\ell_p$-norm for $1\leq p\leq\infty$; in this case, we write $X\oplus_p Y$.

We begin with the following result, which shows that the weak-DD2P is preserved under absolute sums of Banach spaces.

 \begin{proposition}\label{prop:sumabsosuben}
Let $X$ and $Y$ be Banach spaces with the weak-DD2P, and let $N$ be an absolute normalised norm on $\R^2$. Then the absolute sum $X\oplus_N Y$ has the weak-DD2P.
\end{proposition}

\begin{proof}
Let $(x,y)\in B_{X\oplus_N Y}$. Since $\overline{S_{X\oplus_N Y}}^{\omega}=B_{X\oplus_N Y}$ (as $X$ and $Y$ are necessarily infinite-dimensional), we may assume without loss of generality that $\|(x,y)\|_N=1$.

First, consider the case where $x\neq 0$ and $y\neq 0$. Then $\frac{x}{\|x\|}\in S_X$ and $\frac{y}{\|y\|}\in S_Y$. Since both $X$ and $Y$ have the weak-DD2P, there exist nets of super $\Delta$-points $(x_\alpha)_\alpha\subseteq S_X$ and $(y_\beta)_\beta\subseteq S_Y$ converging weakly to $\frac{x}{\|x\|}$ and $\frac{y}{\|y\|}$, respectively. By \cite[Proposition~3.26(1)]{MPR}, every element of the net
\[
(\|x\|x_\alpha,\|y\|y_\beta)_{(\alpha,\beta)}
\]
is a super $\Delta$-point in $X\oplus_N Y$. Since this net converges weakly to $(x,y)$, the desired conclusion follows.

Now consider the case $y=0$ (the case $x=0$ being analogous), which implies $\|x\|=1$. Let $(x_\alpha)_\alpha\subseteq S_X$ be a net of super $\Delta$-points converging weakly to $x$. By \cite[Proposition~3.26(2)]{MPR}, the net $(x_\alpha,0)_\alpha$ consists of super $\Delta$-points in $X\oplus_N Y$ and converges weakly to $(x,0)$.
\end{proof}

We now turn to the special case of the $\ell_\infty$-sum. More precisely, we show that if a Banach space $X$ has the weak-DD2P, then the space $X\oplus_\infty Y$ also has the weak-DD2P for every Banach space $Y$. This result is natural, as it is known that an analogous statement holds for the D2P (see \cite[Lemma~2.1]{lop06}) and for the DD2P \cite[Theorem 2.12]{blr18}. To establish this result, we first require the following auxiliary lemma, whose proof is included for the sake of completeness.

\begin{lemma}\label{lema:superDeltasumainfinito}
Let $X$ and $Y$ be Banach spaces, and let $x\in S_X$ be a super $\Delta$-point. Then, for every $y\in B_Y$, the element $(x,y)$ is a super $\Delta$-point in $X\oplus_\infty Y$.
\end{lemma}

\begin{proof}
    Fix $\varepsilon>0$ and take a net $(x_\alpha)_\alpha\subseteq \Delta_\varepsilon(x)$ weakly convergent to $x$. Then, for every $\alpha$ we have \begin{align*}
        \norm{(x_\alpha,y)-(x,y)}_\infty = \norm{x_\alpha-x}\geq 2-\varepsilon,
    \end{align*}
    so $(x_\alpha,y)\in \Delta_\varepsilon(x,y)$. As $(x_\alpha,y)_\alpha$ converges weakly to $(x,y)$, we conclude that $(x,y)\in \overline{\Delta_\varepsilon(x,y)}^\omega$, as desired.
\end{proof}

Now we can get the desired result.

\begin{proposition}\label{prop:densesuminfty}
    Let $X,Y$ be Banach spaces. If $X$ has the weak-DD2P, then $X\oplus_\infty Y$ also has the weak-DD2P.
\end{proposition}

\begin{proof}
Let $(x,y)\in B_{X\oplus_\infty Y}$. Since $X$ has the weak-DD2P, there exists a net of super $\Delta$-points $(x_\alpha)_\alpha\subseteq S_X$ converging weakly to $x\in B_X$. By Lemma~\ref{lema:superDeltasumainfinito}, it follows that the net $(x_\alpha,y)_\alpha$ consists of super $\Delta$-points in $X\oplus_\infty Y$. Moreover, it is clear that $(x_\alpha,y)_\alpha$ converges weakly to $(x,y)$. Therefore, $X\oplus_\infty Y$ has the weak-DD2P.
\end{proof}

Finally, we obtain results in the opposite direction to Proposition~\ref{prop:sumabsosuben}, namely, we identify the restrictions that the weak-DD2P on an absolute sum $X\oplus_N Y$ imposes on the factor spaces $X$ and $Y$. 

\begin{proposition}\label{prop:sumasabsolutasbaja}
Let $X$ and $Y$ be Banach spaces, and let $N$ be an absolute normalised norm on $\R^2$. Suppose that the absolute sum $X\oplus_N Y$ has the weak-DD2P.
\begin{itemize}
    \item[\upshape(a)] If $N$ is not the $\ell_\infty$-norm, then both $X$ and $Y$ have the weak-DD2P.
    \item[\upshape(b)] If $N$ is the $\ell_\infty$-norm, then at least one of the spaces $X$ or $Y$ has the weak-DD2P.
\end{itemize}
\end{proposition}

For the proof, we shall need the following lemma from \cite{GVM}.

\begin{lemma}\label{lema:superDeltatosummands}
    Let $X,Y$ be Banach spaces and let $N$ be an absolute normalised norm on $\R^2$. Given $x\in S_X$, $y\in S_Y$ and $a,b\geq 0$ with $N(a,b)=1$, suppose that $(ax,by)$ is a super $\Delta$-point in $X\oplus_N Y$.
    \begin{itemize}
        \item[\upshape(a)] If $b\neq1$, then $x$ is a super $\Delta$-point.
        \item[\upshape(b)] If $a\neq 1$, then $y$ is a super $\Delta$-point.
        \item[\upshape(c)] If $a=b=1$, then $x$ or $y$ is a super $\Delta$-point.
    \end{itemize}
\end{lemma}

\begin{proof}[Proof of Proposition \ref{prop:sumasabsolutasbaja}]
    (a). Let $t:=\max\{s\in [0,1] \colon N(s,1)\leq 1  \}$ and observe that $t<1$, because $N$ is not the $\ell_\infty$-norm. Suppose that $X$ fails the weak-DD2P (the same argument works for $Y$). 
    Hence, there exists a non-empty weak open set $U$ in $B_X$ without super $\Delta$-points. Without loss of generality, we may assume that $U$ satisfies the following conditions: 
    \begin{equation*}
    U\subseteq \{x\in B_X\colon \|x\|>t\} \ \text{ and} \ \text{  if $z\in U$, then $\frac{z}{\|z\|}\in U$}. 
    \end{equation*} (Just consider the intersection of the relative weakly open set $\{x\in B_X\colon \|x\|>t\}$ with a finite intersection of slices of $B_X$; if $X$ were finite-dimensional, then $B_X$ contains no super $\Delta$-points, so we can take any relative open set $U$ satisfying the requirements.)  
        Define $$V:=\left\{(z,w)\in B_{X\oplus_N Y} \colon z\in U \right\}$$
        which is a non-empty weak open set in $B_{X\oplus_N Y}$ (it is the inverse image of $U$ by the restriction to $B_{X\oplus_N Y}$ of the canonical projection to the first coordinate, which is weak-to-weak continuous). Hence, there must be some $(ax,by)\in V$ with $x\in S_X$, $y\in S_Y$, $a,b\geq 0$ and $N(a,b)=1$ such that $(ax,by)$ is a super $\Delta$-point. Since $ax\in U$, we obtain $a=\norm{ax}>t$, which yields to $b\neq 1$ from the maximality of $t$. Thanks to Lemma \ref{lema:superDeltatosummands}, we have that $x$ is a super $\Delta$-point. But $x=\frac{ax}{\|ax\|}\in U$, getting thus a contradiction.
        
        (b). Suppose that  both $X$ and $Y$ fail the weak-DD2P. Hence, there are non-empty weak open sets $U$ in $B_X$ and $V$ in $B_Y$ without super $\Delta$-points. Consider $V=U\times W$, which is clearly a non-empty weak open set in $B_{X\oplus_\infty Y} $. Therefore, there is some super $\Delta$-point $(ax,by)$ where $x\in S_X$, $y\in S_Y$, $a,b\geq 0$ and $N(a,b)=1$. We consider three cases
        \begin{enumerate}
            \item If $a=b=1$, then $x=ax\in V$, $y=by\in W$ and $x$ or $y$ is a super $\Delta$-point thanks to Lemma \ref{lema:superDeltatosummands}.
            \item If $a\neq 1$, then $b=1$, $y=by\in W$ and $y$ is a super $\Delta$-point thanks to Lemma \ref{lema:superDeltatosummands}.
            \item If $b\neq 1$, then $a=1$, $x=ax\in V$ and $x$ is a super $\Delta$-point thanks to Lemma \ref{lema:superDeltatosummands}.
        \end{enumerate}
        In every case we obtain a contradiction, so $X$ or $Y$ must have the weak-DD2P.
\end{proof}

\subsection{Köthe-Bochner spaces}\label{subsection:kothe}
For the reader’s convenience, we begin by recalling the presentation of Köthe--Bochner spaces adopted in \cite{Hardtke}. Let $(S,\mathcal{A},\mu)$ be a complete $\sigma$-finite measure space. A \emph{Köthe function space} over $(S,\mathcal{A},\mu)$ is a Banach space $(E,\|\cdot\|_E)$ of real-valued measurable functions, identified up to equality $\mu$-almost everywhere, which satisfies the following conditions:
\begin{enumerate}
    \item $\chi_A \in E$ for every $A\in\mathcal{A}$ with $\mu(A)<\infty$;
    \item every $f\in E$ is $\mu$-integrable over each measurable set $A\in\mathcal{A}$ with $\mu(A)<\infty$;
    \item whenever $f\in E$ and $g$ is a measurable function such that
    \[
    |g(t)|\leq |f(t)| \quad \mu\text{-a.e.},
    \]
    one has $g\in E$ and
    \[
    \|g\|_E\leq \|f\|_E.
    \]
\end{enumerate}

Let $X$ be a Banach space. A function $f\colon S\to X$ is said to be \emph{simple} if there exist finitely many pairwise disjoint sets $A_1,\dots,A_n\in\mathcal{A}$, each of finite measure, such that $f$ is constant on each $A_i$ and vanishes outside $\bigcup_{i=1}^n A_i$.

A function $f\colon S\to X$ is called \emph{Bochner measurable} if there exists a sequence of simple functions $(f_n)_n$ such that
\[
\|f_n(t)-f(t)\|\longrightarrow 0
\quad \mu\text{-a.e.}
\]

Given a Köthe function space $E$ and a Banach space $X$, we denote by $E(X)$ the space of all Bochner measurable functions $f\colon S\to X$ such that the map $t\mapsto\|f(t)\|$ belongs to $E$. Endowed with the norm
\[
\|f\|_{E(X)}=\bigl\|t\mapsto \|f(\cdot)\|\bigr\|_E,
\]
the space $E(X)$ is a Banach space, called the corresponding \emph{Köthe--Bochner space}.
Typical examples are the Lebesgue--Bochner spaces $L_p(\mu,X)$ for $1\leq p<\infty$. We refer the reader to \cite{Lin} for further details. 

It was proved by J.-D.~Hardtke in \cite{Hardtke} that the classical diameter two properties are preserved under the formation of Köthe-Bochner spaces. His proof relies on the use of test families (see \cite[Definition 1.1]{Hardtke}) together with stability results for absolute sums. However, a different argument was developed in \cite{lp} to show that the DD2P (and also the DLD2P and the convex-DLD2P) are likewise stable under these constructions. In what follows, we adopt this latter approach to prove that the weak-DD2P is also preserved when passing to Köthe-Bochner spaces.

We first show that a broad class of functions in $E(X)$ are super $\Delta$-points.

\begin{proposition}\label{prop:superDeltaKotheBochner}
Let $(S,\mathcal{A},\mu)$ be a complete $\sigma$-finite measure space, let $E$ be a Köthe function space over $(S,\mathcal{A},\mu)$, and let $X$ be a Banach space such that the simple functions are dense in $E(X)$. Suppose that $x_1,\ldots,x_n\in S_X$ are super $\Delta$-points, that $A_1,\ldots,A_n\in\mathcal{A}$ are pairwise disjoint measurable sets, and that $\lambda_1,\ldots,\lambda_n>0$ satisfy
\[
\Bigl\|\sum\nolimits_{i=1}^n \lambda_i \chi_{A_i}\Bigr\|_E=1.
\]
Then the function
\[
f:=\sum_{i=1}^n \lambda_i x_i \chi_{A_i}
\]
is a super $\Delta$-point in $E(X)$.
\end{proposition}

\begin{proof}
We will adapt to our setting the proof of \cite[Theorem 4.1 (a)]{lp}. 

Pick $\varepsilon>0$ and $W$ a weak open set of $B_{E(X)}$ containing $f$. We have to find $\widetilde{f}\in W$ with $\|f-\widetilde{f}\|\geq 2-\varepsilon$. We may assume that $W$ is of the form $$W=\left\{ g\in B_{E(X)} \colon |f_1^*(g-f)|<\delta,\ldots,|f_m^*(g-f)|<\delta \right\},$$
   for some $\delta>0$ and $f_1^*,\ldots, f_n^*\in S_{E(X)^*}$. Therefore, define, for each $i\in \{1,\ldots,n\}$ and each $k\in \{1,\ldots,m\}$, the functionals $x_{i,k}^*$ given by 
   $$x_{i,k}^*(x):= f_k^*(x\chi_{A_i}), \quad \forall x\in X.$$
   It is easy to see that $x_{i,k}^*\in X^*$ and $\norm{x_{i,k}^*}\leq \norm{\chi_{A_i}}$, for all $i\in \{1,\ldots,n\}$ and $k\in \{1,\ldots,m\}$. Next, consider 
   $$V_i:= \left\{ x\in B_X \colon |x_{i,1}^*(x-x_i)|<\frac{\delta}{n\lambda_i}, \ldots, |x_{i,m}^*(x-x_i)|<\frac{\delta}{n\lambda_i}\right\},$$ for each $i\in \{1,\ldots,n\}$. They are clearly weak open sets in $B_X$ and satisfy that $x_i\in V_i$ for all $i\in\{1,\ldots,n\}$. Since $x_1,\ldots,x_n$ are super $\Delta$-points, we can find $y_i\in V_i$ with $\norm{x_i-y_i}\geq 2-\varepsilon$, for each $i\in \{1,\ldots,n\}$. Hence, write $\widetilde{f}:=\sum_{i=1}^n \lambda_i y_i \chi_{A_i}$. From $\norm{y_i}\leq1= \norm{x_i}$, it follows that $\|\widetilde{f}(\cdot)\|_E\leq \norm{f(\cdot)}_E$, so $\widetilde{f}\in E(X)$ and $\|\widetilde{f}\|=1$. 

   On the one hand, $$\left|f_k^*(\widetilde{f}-f)\right|= \left|\sum_{i=1}^n \lambda_i x_{i,k}^*(y_i-x_i) \right|\leq \sum_{i=1}^n \lambda_i|x_{i,k}^*(y_i-x_i)|\leq \sum_{i=1}^n \lambda_i\frac{\delta}{n\lambda_i}=\delta,$$
   for all $k\in \{1,\ldots,m\}$, so $\widetilde{f}\in W$. On the other hand, it is straightforward to verify that $\norm{f-\widetilde{f}}\geq 2-\varepsilon$, since $\norm{x_i-y_i}\geq 2-\varepsilon$, for all $i\in \{1,\ldots,n\}$. Thus, $f$ is a super $\Delta$-point.
\end{proof}

Now, we can prove the promised result.

\begin{theorem}\label{thm:Bochner}
    Let $(S,\mathcal A, \mu)$ be a complete $\sigma$-finite measure space, let $E$ be a Köthe function space over $(S,\mathcal A, \mu)$ and let $X$ be a Banach space such that the simple functions are dense in $E(X)$. If $X$ has the weak-DD2P, then $E(X)$ also has the weak-DD2P.
\end{theorem}

\begin{proof}
Let $W$ be a non-empty weakly open subset of $B_{E(X)}$. Since $\overline{S_{E(X)}}^{\omega}=B_{E(X)}$ (as $X$ is infinite-dimensional) and the simple functions are dense in $E(X)$, we may assume that there is some
\[
f=\sum_{i=1}^n x_i\chi_{A_i}\in W\cap S_{E(X)},
\]
where $A_1,\ldots,A_n\in\mathcal{A}$ are pairwise disjoint measurable sets of finite measure and $x_1,\ldots,x_n\in X\setminus\{0\}$.

Since $X$ has the weak-DD2P, for each $i\in\{1,\ldots,n\}$ there exists a net $(x_{i,\alpha_i})_{\alpha_i\in\mathcal{D}_i}\subseteq S_X$ of super $\Delta$-points converging weakly to $\frac{x_i}{\|x_i\|}$. For each $i\in\{1,\ldots,n\}$, consider the linear operator $T_i\colon X\to E(X)$ defined by
\[
T_i(x)=x\chi_{A_i}, \qquad x\in X.
\]
Each $T_i$ is bounded and therefore weak-to-weak continuous. Consequently, the net
\[
\left(\sum_{i=1}^n \|x_i\|\, x_{i,\alpha_i}\chi_{A_i}\right)_
{(\alpha_1,\ldots,\alpha_n)\in \mathcal{D}_1\times\cdots\times\mathcal{D}_n}
\]
converges weakly to $f$ in $E(X)$.

Finally, each element of this net is a super $\Delta$-point by Proposition~\ref{prop:superDeltaKotheBochner}, since
\[
\Bigl\|\sum_{i=1}^n \|x_i\|\,\chi_{A_i}\Bigr\|_E=\|f\|_{E(X)}=1.
\]
This completes the proof.
\end{proof}

In particular, this applies to $L_p(\mu,X)$ for $1\leq p < \infty$. 

\begin{corollary}\label{thm:Bochner-Lp}
Let $(S,\mathcal A, \mu)$ be a complete $\sigma$-finite measure space and let $X$ be a Banach space. If $X$ has the weak-DD2P, then $L_p(\mu,X)$ has the weak-DD2P for $1\leq p< \infty$. 
\end{corollary}

Indeed, the result follows directly from Theorem~\ref{thm:Bochner} as simple functions on $L_p(\mu,X)$ are dense due to Lebesgue Dominated Convergence Theorem. 

\subsection{Projective and symmetric tensor product}\label{subsection:tensores}
The purpose of this subsection is to obtain examples of Banach spaces with the weak-DD2P within the class of (symmetric) projective tensor products of Banach spaces. We begin with the case of projective tensor products, for which we first introduce some notation.

The \emph{projective tensor product} of Banach spaces $X$ and $Y$, denoted by $X\pten Y$, is defined as the completion of the algebraic tensor product $X\otimes Y$ endowed with the norm
\[
\|z\|_{\pi} := \inf \left\{ \sum_{n=1}^k \|x_n\|\,\|y_n\| \colon z=\sum_{n=1}^k x_n\otimes y_n \right\},
\]
where the infimum is taken over all such representations of $z$. It is well known that $\|x\otimes y\|_{\pi}=\|x\|\,\|y\|$ for every $x\in X$ and $y\in Y$, and that the closed unit ball of $X\pten Y$ coincides with the closed convex hull of the set
\[
B_X\otimes B_Y=\{x\otimes y \colon x\in B_X,\ y\in B_Y\}.
\]

Moreover, every operator $G\in L(X,Y^*)$ acts on $X\pten Y$ by
\[
G\Bigl(\sum_{n=1}^k x_n\otimes y_n\Bigr)
= \sum_{n=1}^k G(x_n)(y_n),
\qquad \sum_{n=1}^k x_n\otimes y_n\in X\otimes Y.
\]
This action induces a linear isometry from $L(X,Y^*)$ onto $(X\pten Y)^*$ (see, for instance, \cite[Theorem~2.9]{Ryan}). For further background on diameter two properties and related notions in tensor product spaces, we refer the reader to \cite{aln13,blr4,llr,MartinRueda-symm}.

With the above notation in place, we proceed to obtain examples of tensor product spaces $X\pten Y$ with the weak-DD2P under the assumption that $X$ is a vector-valued function algebra. To this end, we begin with the following proposition, which identifies a class of super $\Delta$-points in this setting.

\begin{proposition}\label{prop:superDeltaTensor}
    Let $X, Y$ be complex Banach spaces and let $A$ be a complex function algebra on a regular Hausdorff topological space $\Omega$. 
    Let $u=\sum_{i=1}^n \lambda_i (f_i\otimes y_i)\in A(\Omega,X)\pten Y$ be a finite-rank tensor with $\lambda_1,\ldots, \lambda_n\geq 0$, $\sum_{i=1}^n \lambda_i=1$, $y_1,\ldots, y_n \in S_Y$ and $f_1,\ldots, f_n\in S_{A(\Omega,X)}$. Suppose that $\norm{\sum_{i=1}^n \lambda_i y_i}=1$ and that there exists $\omega_0\in \SB(A)'$ and $x_0\in S_X$ such that $$f_i(\omega_0)=x_0, \quad \text{ for all } i\in \{1,\ldots,n\}.$$
    Then, $u$ is a super $\Delta$-point.
\end{proposition}

\begin{proof}
    Fix $\varepsilon>0$ and a weakly open set $W$ in $B_{A(\Omega,X)\pten Y}$ with $u\in W$. Pick $x^*\in S_{X^*}$ satisfying that $x^*(x_0)=1$. Since $\Omega$ is Hausdorff, $\omega_0\in \SB(A)'$ and $x^*(f_i(\omega_0))=x^*(x_0)=1$ for all $i\in \{1,\ldots,n\}$, we can find a sequence $(\omega_k)_k\subseteq \SB(A)$ with $\omega_k\neq \omega_j$ for all $k\neq j$, satisfying $|x^*(f_i(\omega_k))-1|<\frac{1}{k}$ for all $k\in \N$ and for all $i\in \{1,\ldots,n\}$. Using again that $\Omega$ is Hausdorff, there is a sequence $(U_k)_k$ formed of pairwise disjoint open sets in $\Omega$ such that $\omega_k\in U_k$ for all $k\in \N$. Furthermore, up to replacing $U_k$ with $ \left(\bigcap_{i=1}^n \left\{\omega\in \Omega\colon |x^*(f_i(\omega))-1|<\frac{1}{k}\right\}\right)\cap U_k$ which is still open and non-empty, we may assume that for each $k\in \N$ we have $$|x^*(f_i(\omega))-1|<\frac{1}{k}, \quad  \forall \omega\in U_k,\ \forall i\in \{1,\ldots,n\}.$$
    Now, invoking Lemma \ref{lemma:engenurysohn} take, for each $k\in \N$, a function $\phi_k\in A$ such that $\phi_k(\omega_k)=1=\norm{\phi_k}$, $\sup_{\omega\in \Omega\setminus U_k}|\phi_k(\omega)|<2^{-k}$ and \begin{equation*}
        |\phi_k(\omega)|+\left(1-2^{-k}\right)|1-\phi_k(\omega)|\leq 1, \quad \forall \omega\in \Omega.
    \end{equation*}
    As we did in the proof of Proposition \ref{prop:superDeltasubalgebraX}, the sequence $(\phi_k)_k$ is weakly convergent to $0$. Hence, defining
    \begin{equation*}
        f_{i,k}:=(1-2^{-k})(1-\phi_k)f_i-\phi_kx_0, \quad \forall k\in \N, \ \forall i\in \{1,\ldots,n\},
    \end{equation*}
    it is easy to see that $(f_{i,k})_k\subseteq B_{A(\Omega,X)}$ and $(f_{i,k})_k$ converges weakly to $f_i$, for each $i\in \{1,\ldots,n\}$. It follows that $(\sum_{i=1}^n \lambda_i (f_{i,k}\otimes y_i))_k$ converges weakly to $u$, and so there exists $m\in \N$ such that $$\sum_{i=1}^n \lambda_i (f_{i,k}\otimes y_i)\in W, \quad \forall k\geq m.$$
    Pick $k_0\in \N$ with $k_0\geq m$ and $\frac{1}{k_0}<\varepsilon$. Therefore, we obtain at the same time that \begin{equation*}
        |x^*(f_i(t_{k_0}))-1|<\frac{1}{k_0}<\varepsilon \quad \text{ and } \quad x^*(f_{i,k_0}(t_{k_0}))=-1,
    \end{equation*}
    for all $i\in \{1,\ldots,n\}$. Finally, taking $y^*\in S_{Y^*}$ with $y^*\left(\sum_{i=1}^n \lambda_i y_i \right)=1$, we have
    \begin{align*}
        \norm{u-u_0}&\geq |(T\otimes y^*)(u-u_0)|\\&= \left| \sum_{i=1}^n \lambda_i x^*(f_i(t_{k_0}))y^*(y_i)-  \sum_{i=1}^n \lambda_i x^*(f_{i,k_0}(t_{k_0}))y^*(y_i)\right| \\&=\left| \sum_{i=1}^n \lambda_i x^*(f_i(t_{k_0}))y^*(y_i)+1\right|\\&\geq 2- \left| \sum_{i=1}^n \lambda_i x^*(f_i(t_{k_0}))y^*(y_i)- 1\right|\\&\geq 2-  \sum_{i=1}^n \lambda_i |x^*(f_i(t_{k_0}))- 1| \geq 2-\varepsilon,
    \end{align*}
    where $T\in S_{A(\Omega,X)^*}$ is given by $T(f)=x^*(f(t_{k_0}))$ for all $f\in A(\Omega,X)$. Observe that we have used $y^*(y_i)=1$ for all $i\in \{1,\ldots,n\}$, which follows from the equality $y^*(\sum_{i=1}^n \lambda_i y_i)=1$. Thus, $u$ is a super $\Delta$-point.
\end{proof}

Now, we are able to establish one of the main results in this section.

\begin{theorem}\label{theo:maintheotensornosim}
    Let $X$ be a complex Banach space and let $A$ be a complex function algebra on a regular Hausdorff topological space $\Omega$. Suppose that the set $\SB(A)'\cap \SB(A)$ is infinite and let $Y$ be a complex Banach space 
    satisfying one of the following conditions:
    \begin{itemize}
        \item[\upshape(a)] Every convex combination of non-empty relatively weakly open subsets of $B_Y$ intersects $S_Y$.
        \item[\upshape(b)] There exists $y^*\in S_{Y^*}$ such that $\acconv\{y\in B_Y\colon y^*(y)=1\}=B_Y$.
    \end{itemize}
    Then, $A(\Omega,X)\pten Y$ has the weak-DD2P.
\end{theorem}

\begin{proof}
    Let $W$ be a non-empty weak open set in $B_{A(\Omega,X)\pten Y}$ and we have to find some super $\Delta$-point in $W$. First, suppose that there is some $u=\sum_{i=1}^n \lambda_i (f_i\otimes y_i)\in W$ with $\lambda_1,\ldots, \lambda_n\geq 0$, $\sum_{i=1}^n \lambda_i =1$, and $\norm{f_i}=\norm{y_i}=1$ for all $i\in \{1,\ldots,n\}$. Moreover, suppose that $\norm{\sum_{i=1}^n \lambda_i y_i}=1$.

    Under these assumptions, we can find a super $\Delta$-point in $W$. Indeed, pick a sequence $(\omega_k)_k\subseteq \SB(A)'\cap \SB(A)$ with $\omega_k\neq \omega_j$ for all $k\neq j$. Since $\Omega$ is Hausdorff, we can find a sequence $(U_k)_k$ of pairwise disjoint open sets in $\Omega$ such that $\omega_k\in U_k$ for all $k\in \N$. Thanks to Lemma \ref{lemma:engenurysohn}, we can find, for each $k\in \N$, a function $\phi_k\in A$ such that $\phi_k(\omega_k)=1=\norm{\phi_k}$, $\sup_{\omega\in \Omega\setminus U_k} |\phi_k(\omega)|<2^{-k}$ and $$|\phi_k(\omega)|+(1-2^{-k})|1-\phi_k(\omega)|\leq 1, \quad \forall \omega\in \Omega.$$ Once again, $(\phi_k)_k$ converges weakly to $0$. Next, fix some $x_0\in S_X$ and define 
    $$f_{i,k}:= (1-2^{-k})(1-\phi_k)f_i+\phi_kx_0, \quad \forall k\in \N.$$
    It follows that $(f_{i,k})_k\subseteq B_{A(\Omega,X)}$ and it converges weakly to $f$. As a consequence, the sequence $(\sum_{i=1}^n \lambda_i (f_{i,k}\otimes y_i))_k$ converges weakly to $u$ and so there is some $k_0$ satisfying that $u_0=\sum_{i=1}^n \lambda_i (f_{i,k_0}\otimes y_i)\in W$. Finally, $f_{i,k_0}(\omega_0)=x_0$ for all $i\in \{1,\ldots,n\}$, so $u_0$ is a super $\Delta$-point, thanks to Proposition \ref{prop:superDeltaTensor}.

    Thus, we have just to guarantee the existence of such  $u$ in $W$, whenever $Y$ satisfies (a) or (b).

    (a). Since $B_{A(\Omega,X)\pten Y}=\cconv\{f\otimes y \colon f\in S_{A(\Omega,X)}, \ y\in S_Y\}$, and since weak open sets are also norm open sets, there is some  $\widetilde{u}=\sum_{i=1}^n \lambda_i (f_i\otimes y_i)\in W$, with $\lambda_1,\ldots, \lambda_n \geq 0$, $\sum_{i=1}^n \lambda_i=1$ and $\norm{f_i}=\norm{y_i}=1$ for all $i\in \{1,\ldots,n\}$. Next, consider the net $(U_\alpha)_\alpha$ formed of weak open neighbourhoods of $0$, and define for each $\alpha$, the convex combination of weak open sets $$V_\alpha =\sum_{i=1}^n \lambda_i (y_i+U_\alpha).$$
        Using the assumption we can find $\widetilde{y}_\alpha \in V_\alpha \cap S_Y$, for each $\alpha$. Thus, $$\widetilde{y}_\alpha = \sum_{i=1}^n \lambda_i(y_i+z_\alpha)$$
        for each $\alpha$, where $(z_\alpha)_\alpha$ is  weakly convergent to $0$. Hence, $(\sum_{i=1}^n \lambda_i(f_i\otimes (y_i+z_\alpha)))_\alpha$ converges weakly to $\widetilde{u}$, and so there is some $\alpha_0$ such that $u=\sum_{i=1}^n \lambda_i (f_i\otimes (y_i+z_{\alpha_0}))\in W$. The conclusion follows since $\norm{\sum_{i=1}^n \lambda_i (y_i+z_{\alpha_0})}=\norm{\widetilde{y}_{\alpha_0}}=1$.
        
      (b). Using again that $B_{A(\Omega,X)\pten Y}=\cconv\{f\otimes y \colon f\in S_{A(\Omega,X)}, \ y\in S_Y\}$, we can find some  $u=\sum_{i=1}^n \lambda_i (f_i\otimes y_i)\in W$, such that $\lambda_1,\ldots, \lambda_n \geq 0$, $\sum_{i=1}^n \lambda_i=1$ and $\norm{f_i}=\norm{y_i}=1$ for all $i\in \{1,\ldots,n\}$. Furthermore, by assumptions there is some $y^*\in S_{Y^*}$ such that $$B_Y=\acconv\{ y\in S_Y \colon y^*(y)=1\}.$$
        Therefore, we may assume that $y_i\in \{ y\in S_Y \colon y^*(y)=1\}$ for all $i\in \{1,\ldots,n\}$, which yields to 
        \[1\geq\norm{\sum_{i=1}^n \lambda_i y_i}\geq \left|\sum_{i=1}^n \lambda_i y^*(y_i)\right|=1.\qedhere\]
\end{proof}

\begin{remark}
Observe that in the proofs of Proposition~\ref{prop:superDeltaTensor} and of Theorem~\ref{theo:maintheotensornosim} we can replace the use of Lemma \ref{lemma:engenurysohn} with the employment of Urysohn's lemma (in the same fashion as we perform in Proposition~\ref{prop:C0(L,X)}) to obtain the following version of Theorem~\ref{theo:maintheotensornosim} for $C_0(L,X)$ spaces in the following terms.
\end{remark}

\begin{theorem}\label{theo:maintheotensornosimc0case}
Let $X$ be a Banach space and let $L$ be a locally compact Hausdorff space. Assume that $L'$ is infinite is infinite and let $Y$ be a Banach space satisfying one of the following conditions:
    \begin{itemize}
        \item[\upshape(a)] Every convex combination of non-empty relatively weakly open subsets of $B_Y$ intersects $S_Y$.
        \item[\upshape(b)] There exists $y^*\in S_{Y^*}$ such that $\acconv\{y\in B_Y\colon y^*(y)=1\}=B_Y$.
    \end{itemize}
    Then, $C_0(L,X)\pten Y$ has the weak-DD2P.
\end{theorem}

Before turning to the symmetric projective tensor product, we illustrate in the following remark several examples of Banach spaces $Y$ for which Theorem~\ref{theo:maintheotensornosim} applies.

\begin{remark}\label{remark:examaplitensornosi} Let us present some comments on the results above.
\begin{itemize}
    \item [\upshape(a)] Following the notation of \cite{lmr}, a Banach space $X$ is said to have the \textit{property (CS)} if every convex combination of slices of $B_X$ intersects $S_X$. It was proved in \cite[Theorem 3.4]{lmr} that a Banach space $X$ has property (CS) if, and only if, given any convex combination of slices $C$ of $B_X$ there are $x,y\in C$ such that $\Vert x-y\Vert=2$. It is known that the fact that every convex combination of slices of the unit ball has diameter 2 does not imply the property (CS) \cite[Example 3.3]{lmr}. Examples of Banach space satisfying the property CS are $L_1([0,1])$ \cite[Example 3.2]{al18}, $C(K)$ for any infinite compact space $K$ or square Banach spaces from \cite{acllrz23}. For background about spaces with property (CS) and properties around it we refer the reader to \cite{al18,hkp2,lmr}.

Given a Banach space $X$, it is known, by a result of Bourgain, that every non-empty relatively weakly open subset of $B_X$ contains a convex combination of slices of $B_X$ (cf.\ e.g.\ \cite[Lemma~2.2]{MPR}). Consequently, the assumption described in part~(a) of Theorem~\ref{theo:maintheotensornosim} is precisely the so-called property~(CS) of the space $Y$.

    \item [\upshape(b)] Given a Banach space $X$ and an element $z\in S_X$, we say that $z$ is a \emph{spear vector} if
\[
\max_{\theta\in\mathbb{T}} \|z+\theta x\| = 1+\|x\|
\quad \text{for all } x\in X.
\]
Moreover, when dealing with dual spaces, it is known that an element $z^*\in S_{X^*}$ is a spear vector in $X^*$ if and only if
\[
B_X = \overline{\aconv}\{ x\in B_X \colon z^*(x)=1\}
\]
(see, for instance, \cite[Theorem~2.9]{SpearsBook}). Consequently, condition~(b) in Theorem~\ref{theo:maintheotensornosim} is equivalent to the existence of a spear vector in $Y^*$.

Typical examples of Banach spaces whose duals admit spear vectors include almost CL-spaces and separable lush spaces. For definitions and further background on this topic, we refer the reader to \cite{SpearsBook}.

    \item [\upshape(c)] Let us observe that the properties described in parts~(a) and~(b) above are independent. Indeed, consider first the space $X=\ell_1$. On the one hand, $\ell_1$ fails property~(CS), since its unit ball contains slices of arbitrarily small diameter. On the other hand, we have
\[
\overline{\aconv}\{y\in \ell_1 \colon \mathbf{1}(y)=1\}=B_{\ell_1},
\]
where $\mathbf{1}$ denotes the sequence $(1,1,1,\ldots)\in \ell_\infty=\ell_1^*$.

Conversely, consider the space $X=L_1([0,1],\ell_2)$. By \cite[Corollary~4.22]{SpearsBook} and the fact that $\ell_2$ does not contain any spear vector, there exists no functional $x^*\in X^*$ which is a spear vector. Nevertheless, $X$ satisfies property~(CS). Indeed, let
\[
C=\sum_{i=1}^n \lambda_i S(B_X,g_i,\alpha)
\]
be a convex combination of slices of $B_X$, where $g_i\in X^*=L_\infty([0,1],\ell_2)$ and $\|g_i\|=1$ for each $i$. Then there exist pairwise disjoint measurable subsets $B_1,\ldots,B_n\subseteq [0,1]$ such that
\[
B_i\subseteq \{t\in[0,1]\colon \|g_i(t)\|>1-\alpha\}
\quad\text{and}\quad \mu(B_i)>0,
\]
for every $i\in\{1,\ldots,n\}$. By \cite[Theorem~I.4.2]{DGZ}, there exists a Borel mapping
\[
D_\infty\colon \ell_2\to\ell_2
\]
such that $\langle D_\infty(y),y\rangle=\|y\|^2=\|D_\infty(y)\|^2$ for all $y\in B_{\ell_2}$. For each $i\in\{1,\ldots,n\}$, define
\[
\widetilde g_i(t)=
\begin{cases}
\dfrac{g_i(t)}{\|g_i(t)\|}, & t\in B_i,\\[1ex]
0, & t\notin B_i.
\end{cases}
\]
Then $D_\infty\circ \widetilde g_i\in L_1([0,1],\ell_2)$ and $\|D_\infty\circ \widetilde g_i\|=\mu(B_i)$. Setting
\[
f_i:=\frac{1}{\mu(B_i)}(D_\infty\circ \widetilde g_i),
\]
we obtain
\begin{align*}
\langle f_i,g_i\rangle
&=\frac{1}{\mu(B_i)}\int_{B_i}\|g_i(t)\|
   \langle D_\infty(\widetilde g_i(t)),\widetilde g_i(t)\rangle\,dt \\
&=\frac{1}{\mu(B_i)}\int_{B_i}\|g_i(t)\|\,dt
>1-\alpha.
\end{align*}
Hence, $\sum_{i=1}^n \lambda_i f_i\in C$, and moreover,
\[
\Bigl\|\sum_{i=1}^n \lambda_i f_i\Bigr\|=1,
\]
since the supports of the functions $f_i$ are pairwise disjoint by construction. Therefore, $C\cap S_X\neq\emptyset$, and thus $X$ has property~(CS).

\end{itemize}
\end{remark}

We now turn to projective symmetric tensor products, for which some additional notation will be required.

For $N\in\mathbb{N}$, we denote by $\mathcal{P}(^N X)$ the Banach space of $N$-homogeneous scalar-valued continuous polynomials on a Banach space $X$. Endowed with the norm
\[
\|P\|=\sup_{x\in B_X} |P(x)|, \qquad P\in\mathcal{P}(^N X),
\]
the space $\mathcal{P}(^N X)$ becomes a normed space.

Given a Banach space $X$, the \emph{$N$-fold projective symmetric tensor product} of $X$, denoted by $\widehat{\otimes}_{\pi,s,N} X$, is defined as the completion of the symmetric tensor product $\otimes^{s,N}X$ with respect to the norm
\[
\|u\|:=\inf\left\{
\sum_{i=1}^n |\lambda_i|\,\|x_i\|^N \colon
u=\sum_{i=1}^n \lambda_i x_i^N,\ n\in\mathbb{N},\ x_i\in X
\right\}.
\]
It is well known that
\[
B_{\widehat{\otimes}_{\pi,s,N} X}
=\overline{\aconv}\bigl(\{x^N \colon x\in S_X\}\bigr)
\quad\text{and}\quad
\bigl(\widehat{\otimes}_{\pi,s,N} X\bigr)^*=\mathcal{P}(^N X)
\]
(see \cite{flo} for background).

For further information on projective symmetric tensor products in connection with diameter two properties and related notions, we refer the reader to
\cite{ab,blrasq,llr,MartinRueda-symm}.

We begin by exhibiting a class of elements in the unit ball of $\widehat{\otimes}_{\pi,s,N} C(K)$ which are super $\Delta$-points.

\begin{proposition}\label{prop:superDeltasymmetrictensor}
Let $K$ be a compact Hausdorff topological space and let $N\in \N$. Pick $u=\sum_{i=1}^n \lambda_i f_i^N\in \widehat{\otimes}_{\pi,s,N}C(K)$ with $\lambda_1,\ldots,\lambda_n\geq 0$ and $\sum_{i=1}^n \lambda_i=1$. If there exists $t_0\in K'$ such that $$f_i(t_0)=1=\norm{f_i}, \quad \forall i\in \{1,\ldots,n\},$$
    then $u$ is a super $\Delta$-point.
\end{proposition}

\begin{proof}
     Fix $\varepsilon>0$ and a weak open set $W\subseteq B_{\widehat{\otimes}_{\pi,s,N}C(K)}$ with $u\in W$. Since $K$ is Hausdorff, $t\in K'$ and $f_i(t)=1$ for all $i\in \{1,\ldots,n\}$, we can find a sequence $(t_k)_k\subseteq K$ with $t_k\neq t_j$ for all $k\neq j$ and satisfying  $|f_i(t_k)-1|<\frac{1}{k}$ for all $k\in \N$ and for all $i\in \{1,\ldots,n\}$. Using again that $K$ is Hausdorff, we can find a sequence $(V_k)_k$ formed of pairwise disjoint open sets in $K$ such that $t_k\in V_k$ for all $k\in \N$. Furthermore, up to replacing $V_k$ with $ \left(\bigcap_{i=1}^n \left\{t\in K\colon |f_i(t)-1|<\frac{1}{k}\right\}\right)\cap V_k$ which is still open and non-empty, we may assume that for each $k\in \N$ we have $$|f_i(t)-1|<\frac{1}{k}, \quad  \forall t\in V_k,\ \forall i\in \{1,\ldots,n\}.$$ Now, take a sequence of Urysohn functions $(\varphi_k)_k$ satisfying $\varphi_k(K)\subseteq[0,1]$,  $\varphi_k(t_k)=1$ and $\varphi_k|_{K\setminus V_k}=0$, for each $k\in \N$. Next, consider the following sequences, formed by splitting the previous ones. Write, for each $k\in \N$,  \begin{align*}
         \widetilde{U}_k&:=V_{2k}, &\quad \phi_k&:=\varphi_{2k}, &\quad r_k&:=t_{2k},
         \\ \widetilde{V}_k&:=V_{2k+1}, &\quad \psi_k&:=\varphi_{2k+1}, &\quad s_k&:=t_{2k+1}
     \end{align*}
     Now, define
     \begin{equation*}
         f_{i,k}:=(1-\phi_k)(1-\psi_k)f_i+\phi_k-\psi_k\in C(K), \quad \forall k\in \N, \ \forall i\in \{1,\ldots,n\}.
     \end{equation*}
     Notice that for all $k\in \N$ and for all $i\in \{1,\ldots,n\}$, we have $$|f_{i,k}(t)|\leq (1-\phi_k(t))|f(t)|+\phi_k(t)\leq 1, \quad \forall t\in K\setminus \widetilde{V}_k,$$ and
     $$|f_{i,k}(t)|  \leq (1-\psi_k(t))|f(t)|+\psi_k(t)\leq 1, \quad \forall t\in  \widetilde{V}_k.$$
    Thus, $f_{i,k}\in B_{C(K)}$ for all $k\in \N$ and for all $i\in \{1,\ldots,n\}$. Finally, the pairwise disjointness of the sets $V_k$ (and hence of the sets $\widetilde{U}_k$ and $\widetilde{V}_k$) implies that, for each $i\in\{1,\ldots,n\}$, the sequence $(f_{i,k})_k$ converges weakly to $f_i$. Furthermore, since $C(K)$ has the Dunford-Pettis property, it has the polynomial Dunford-Pettis property (see \cite[Theorem 2.1]{ryan79}) and so  $\left(\sum_{i=1}^n \lambda_if_{i,k}^N\right)_k$ converges weakly to $u$. Thus, there exists $m\in \N$ such that $$\sum_{i=1}^n \lambda_i f_{i,k}^N\in W, \quad \forall k\geq m.$$ Pick $k_0\in \N$ with $k_0\geq m$ and $\frac{1}{k_0}<\frac{\varepsilon}{N}$. For each $i\in \{1,\ldots,n\}$, we obtain, on the one hand, \begin{equation}\label{Eq7}
        |f_i(r_{k_0})-1|<\frac{1}{k_0}<\frac{\varepsilon}{N} \quad \text{ and } \quad |f_i(s_{k_0})-1|<\frac{1}{k_0}<\frac{\varepsilon}{N},
    \end{equation}
    and on the other hand,
    \begin{equation}\label{Eq8}
        f_{i,k_0}(r_{k_0})=1 \quad \text{ and } \quad f_{i,k_0}(s_{k_0})=-1.
    \end{equation}  Finally, writing $u_0=\sum_{i=1}^n \lambda_i f_{i,k_0}^N$, and using (\ref{Eq7}) and (\ref{Eq8}), we conclude
    \begin{align*}
        \norm{u-u_0}&\geq |(\delta_{s_{k_0}} \delta_{r_{k_0}}^{N-1})(u-u_0)|\\&=\left|\sum_{i=1}^n \lambda_if_i(s_{k_0}) f_i(r_{k_0})^{N-1}-\sum_{i=1}^n \lambda_if_{i,k_0}(s_{k_0}) f_{i,k_0}(r_{k_0})^{N-1} \right|\\&=\left|\sum_{i=1}^n \lambda_if_i(s_{k_0}) f_i(r_{k_0})^{N-1}+1\right|\\&\geq 2-\left|\sum_{i=1}^n \lambda_if_i(s_{k_0}) f_i(r_{k_0})^{N-1}-1\right| \\&\geq2-\sum_{i=1}^n \lambda_i|f_i(s_{k_0}) f_i(r_{k_0})^{N-1}-1|\\&\geq  2- \sum_{i=1}^n \lambda_i (|f_i(s_{k_0}) -1|+(N-1)|f_i(r_{k_0}) -1|)\geq 2-\varepsilon,
    \end{align*}
    where $\delta_{s_{k_0}} \delta_{r_{k_0}}^{N-1}$ is the norm-one functional in $\widehat{\otimes}_{\pi,s,N}C(K)$ given by $(\delta_{s_{k_0}} \delta_{r_{k_0}}^{N-1})(f^N)=f(s_{k_0})f(r_{k_0})^{N-1}$ for all $f\in C(K)$. Thus, $u$ is a super $\Delta$-point as desired.
\end{proof}

With the previous result in mind we show that the weak-DD2P is inherited by symmetric projective tensor products of $C(K)$ spaces.

\begin{theorem}\label{theo:proyectivosimetrico}
    Let $K$ be a compact Hausdorff topological space, and let $N\in \N$. If $K'$ is infinite, then $\widehat{\otimes}_{\pi,s,N} C(K)$ has the weak-DD2P.
\end{theorem}

\begin{proof}
     Let $W\subseteq B_{\widehat{\otimes}_{\pi,s,N} C(K)}$ be a non-empty weak open set and we are going to find some super $\Delta$-point in $W$. Since $B_{\widehat{\otimes}_{\pi,s,N} C(K)}=\acconv \{f^N \colon f\in S_{C(K)}\}$, it follows that there is some $u=\sum_{i=1}^n \lambda_if_i^N\in W$, where  $\sum_{i=1}^n |\lambda_i|=1$ (we may assume that $\lambda_1,\ldots, \lambda_n \neq 0$) and $f_1,\ldots, f_n\in S_{C(K)}$. Now, take a sequence $(t_k)_k\subseteq K'$ with $t_k\neq t_j$ for all $k\neq j$, and find a sequence $(U_k)_k$ of pairwise disjoint open sets in $K$ such that $t_k\in U_k$ for all $k\in \N$. Next, define for each $k\in \N$, a Urysohn function $\varphi_k \colon K\rightarrow[0,1]$ satisfying $\varphi_k(t_k)=1$ and $\varphi_k|_{K\setminus U_k}=0$. Hence, consider for each $i\in \{1,\ldots,n\}$ the sequence of functions $(f_{i,k})_k$ given by $$f_{i,k}=\frac{\lambda_i}{|\lambda_i|}(1-\varphi_k)f_i+\varphi_k\in C(K).$$
    It is easy to see that $\norm{f_{i,k}}\leq 1$ for all $k\in \N$ and for all $i\in \{1,\ldots,n\}$. Furthermore, 
    $$ f_{i,k}(t_k)=1, \quad \forall k\in \N, \ \forall i\in \{1,\ldots,n\}.$$
    Since the sets $U_k$ are pairwise disjoint, it is clear that $(f_{i,k})_k$ converges weakly to $\frac{\lambda_i}{|\lambda_i|}f_i$ for each $i\in \{1,\ldots,n\}$. By the same argument using the Dunford-Pettis property in $C(K)$ of Proposition~\ref{prop:superDeltasymmetrictensor}, we get that $(f_{i,k}^N)$ converges weakly to $\frac{\lambda_i}{\vert \lambda_i\vert}f_i^N$ in $\widehat{\otimes}_{\pi,s,N} C(K)$. Thus, there exists $k_0\in \N$ such that $z=\sum_{i=1}^n |\lambda_i| f_{i,k_0}^N \in W$. Finally, $z$ is a super $\Delta$-point by virtue of Proposition \ref{prop:superDeltasymmetrictensor}.
\end{proof}

\section{Characterization of the weak-DD2P in \texorpdfstring{$L_1(\mu,X)$}{L1(mu,X)}, \texorpdfstring{$L_\infty(\mu,X)$}{Linfty(mu,X)}, and \texorpdfstring{$C(K,X)$}{C(K,X)}}\label{section:espaclavectorval}

In this section, we aim to provide a precise description of the weak-DD2P for vector-valued versions of the spaces $L_1(\mu)$, $L_\infty(\mu)$, and $C(K)$. We begin with the characterisation of the weak-DD2P for the space $L_1(\mu,X)$.

\begin{theorem}
Let $(S,\mathcal{A},\mu)$ be a complete $\sigma$-finite measure space and let $X$ be a Banach space. Then the following statements are equivalent:
\begin{itemize}
    \item[\upshape(i)] The space $L_1(\mu,X)$ has the weak-DD2P.
    \item[\upshape(ii)] The measure $\mu$ is atomless or the Banach space $X$ has the weak-DD2P.
\end{itemize}
\end{theorem}

\begin{proof}
(ii)$\Rightarrow$(i). If $\mu$ is atomless, then $L_1(\mu,X)$ has the Daugavet property (see, for instance, \cite[Theorem 3.4.4]{DaugavetBook}, for instance), and hence it has the DD2P; in particular, it has the weak-DD2P. (Alternatively, one may apply \cite[Theorem~4.8]{MPR}.) If, on the other hand, $X$ has the weak-DD2P, then the conclusion follows from Corollary~\ref{thm:Bochner-Lp}.

(i)$\Rightarrow$(ii). It is a folklore result, as a consequence of the decomposition of a positive measure into its atomic and non-atomic parts (see \cite[Theorem~2.1]{Johnson}), that
\[
L_1(\mu,X)\cong L_1(\nu,X)\oplus_1 \ell_1(I,X),
\]
where $\nu$ is atomless (possibly null) and $I$ is a (possibly empty) index set. By Proposition~\ref{prop:sumasabsolutasbaja}, if $I\neq\emptyset$ (that is, if $\mu$ is not atomless), then $X$ must have the weak-DD2P.
\end{proof}

Our next goal is to provide a characterisation of the weak-DD2P for the space $L_\infty(\mu,X)$. To this end, we first establish the following proposition, which shows that $\ell_\infty(I,X)$ has the weak-DD2P for every Banach space $X$ whenever the index set $I$ is infinite.

\begin{proposition}\label{prop:denselinftyX}
Let $I$ be an infinite set and let $X$ be a non-zero Banach space. Then $\ell_\infty(I,X)$ has the weak-DD2P.
\end{proposition}

\begin{proof}
Let $x\in B_{\ell_\infty(I,X)}$. Since $I$ is infinite, we can choose a sequence $(A_n)_n$ of pairwise disjoint infinite (countable) subsets of $I$. Fix $x_0\in S_X$ and define a sequence $(y_n)_n\subseteq \ell_\infty(I,X)$ by
\[
y_n(i):=
\begin{cases}
x(i), & i\notin A_n,\\
x_0, & i\in A_n.
\end{cases}
\]
Clearly, $\|y_n\|=1$ for all $n\in\mathbb{N}$.

We claim that $(y_n)_n$ converges weakly to $x$. To see this, define a linear operator
\[
T\colon c_0\to \ell_\infty(I,X), \qquad T(e_n)=y_n-x \quad (n\in\mathbb{N}),
\]
where $(e_n)_n$ denotes the canonical basis of $c_0$. For scalars $\lambda_1,\ldots,\lambda_n$, using that the supports of $y_i-x$ for $1\leq i\leq n$ are pairwise disjoint, we obtain
\[
\Bigl\|T\Bigl(\sum_{i=1}^n \lambda_i e_i\Bigr)\Bigr\|
=\Bigl\|\sum_{i=1}^n \lambda_i (y_i-x)\Bigr\|
=\max_{1\leq i\leq n}\|\lambda_i x_0\|
=\max_{1\leq i\leq n}|\lambda_i|.
\]
Hence $\|T\|=1$, so $T$ is weak-to-weak continuous, and therefore $(y_n)_n$ converges weakly to $x$.

It remains to show that each $y_n$ is a super $\Delta$-point. Fix $n\in\mathbb{N}$ and write $A_n=\{\alpha_m\colon m\in\mathbb{N}\}$. Define a sequence $(z_m)_m\subseteq \ell_\infty(I,X)$ by
\[
z_m(i):=
\begin{cases}
y_n(i), & i\neq \alpha_m,\\
-\,x_0, & i=\alpha_m.
\end{cases}
\]
Then $z_m-y_n=(-2x_0)\,e_{\alpha_m}$ for all $m\in\mathbb{N}$. Consequently, $(z_m)_m$ converges weakly to $y_n$ (with the same argument than the one above) and $\|z_m-y_n\|=2$ for all $m$. By \cite[Proposition~3.4]{MPR}, this shows that $y_n$ is a super $\Delta$-point. Since $n$ is arbitrary, the proof is complete.
\end{proof}

We are now able to present the main result on vector-valued $L_\infty$ spaces.

\begin{theorem}
Let $(S,\mathcal{A},\mu)$ be a measure space and let $X$ be a Banach space. Then the following statements are equivalent:
\begin{itemize}
    \item[\upshape(i)] The space $L_\infty(\mu,X)$ has the weak-DD2P.
   \item[\upshape(ii)] At least one of the following conditions holds:
$\mu$ is not purely atomic, $\mu$ has infinitely many atoms, or $X$ has the weak-DD2P.
\end{itemize}
\end{theorem}

\begin{proof}
(ii)$\Rightarrow$(i). It is a folklore result, as a consequence of the decomposition of a measure into its atomic and non-atomic parts (see \cite[Theorem~2.1]{Johnson}), that
\[
L_\infty(\mu,X)\cong L_\infty(\nu,X)\oplus_\infty \ell_\infty(I,X),
\]
where $\nu$ denotes the atomless part of $\mu$ and $I$ indexes the atoms of $\mu$.

If $\nu\neq 0$ (that is, $\mu$ is not purely atomic), then $L_\infty(\nu,X)$ has the Daugavet property (see \cite[Corollary~4.13]{lrt}), and hence $L_\infty(\mu,X)$ has the weak-DD2P by Proposition~\ref{prop:densesuminfty}. If the set $I$ is infinite, then $\ell_\infty(I,X)$ has the weak-DD2P for every Banach space $X$ by Proposition~\ref{prop:denselinftyX} and the result follows again from Proposition~\ref{prop:densesuminfty}. Finally, if $X$ has the weak-DD2P, the only remaining case is that $I$ is not empty and finite and hence $\ell_\infty(I,X)$ has the weak-DD2P by Propositions~\ref{prop:densesuminfty}. Therefore, the same proposition gives that the whole space $L_\infty(\mu,X)$ does.

(i)$\Rightarrow$(ii). Suppose that $L_\infty(\mu,X)$ has the weak-DD2P. By the above decomposition and Proposition~\ref{prop:sumasabsolutasbaja}, at least one of the summands $L_\infty(\nu,X)$ or $\ell_\infty(I,X)$ must have the weak-DD2P. If $\nu\neq 0$, then $\mu$ is not purely atomic. If $I$ is infinite, then $\mu$ has infinitely many atoms. Finally, if $I$ is non-empty and finite, then $\ell_\infty(I,X)$ has the weak-DD2P if and only if $X$ has the weak-DD2P by Propositions \ref{prop:densesuminfty} and \ref{prop:sumasabsolutasbaja}. 
\end{proof}

We conclude this section with a description of the weak-DD2P for spaces of the form $C(K,X)$.

\begin{theorem}
Let $K$ be a compact Hausdorff topological space and let $X$ be a Banach space.
\begin{enumerate}
    \item[\upshape(a)] If $K'=\emptyset$ (equivalently, if $K$ is finite), then $C(K,X)$ has the weak-DD2P if and only if $X$ has the weak-DD2P.
    
    \item[\upshape(b)] If $K'$ is non-empty and finite, then $C(K,X)$ has the weak-DD2P if and only if $X$ is infinite-dimensional.
    
    \item[\upshape(c)] If $K'$ is infinite (equivalently, $K''\neq\emptyset$), then $C(K,X)$ always has the weak-DD2P.
\end{enumerate}
\end{theorem}

\begin{proof}
(a). If $K$ is finite, then
\[
C(K,X)=\underbrace{X\oplus_\infty\cdots\oplus_\infty X}_{n\ \text{times}},
\]
where $n$ is the cardinal of $K$. Hence, the conclusion follows from Propositions~\ref{prop:densesuminfty} and~\ref{prop:sumasabsolutasbaja}.

(b). Assume first that $X$ is finite-dimensional, so that the norm of $X$ is weakly continuous. As $K'$ is finite, we may consider the non-empty weakly open subset
\[
W:=\bigcap_{t\in K'}\Bigl\{f\in C(K,X)\colon \|f(t)\|<\tfrac12\Bigr\}.
\]
By \cite[Theorem~4.8]{lrt}, the set $W$ contains no super $\Delta$-points, since $B_X$ has no super $\Delta$-points and, by construction, functions in $W$ cannot attain their norm at any point of $K'$. Consequently, $C(K,X)$ fails the weak-DD2P.

Conversely, assume that $X$ is infinite-dimensional and fix $t_0\in K'$. Since $K$ is Hausdorff, there exists an open set $U_0\subseteq K$ such that $U_0\cap K'=\{t_0\}$. Choose an open neighbourhood $U$ of $t_0$ satisfying
\[
t_0\in U\subseteq \overline{U}\subseteq U_0.
\]
Let $(U_\beta)_\beta$ denote the net of all open neighbourhoods of $t_0$ contained in $U$, ordered by reverse inclusion.

Fix $f\in B_{C(K,X)}$ and set $a:=f(t_0)$. Since $\dim X=\infty$, there exists a net $(x_\alpha)_\alpha\subseteq S_X$ converging weakly to $a$. Define a net $(f_{\alpha,\beta})_{(\alpha,\beta)}\subseteq C(K,X)$ by
\[
f_{\alpha,\beta}(t)=
\begin{cases}
f(t), & t\notin U_\beta,\\
x_\alpha, & t\in U_\beta.
\end{cases}
\]
Each function $f_{\alpha,\beta}$ is continuous, since points of $U_\beta$ may accumulate only at $t_0$. Moreover,
\[
\|f_{\alpha,\beta}(t_0)\|=1=\|f\|,
\]
and therefore $f_{\alpha,\beta}$ is a super $\Delta$-point for every $(\alpha,\beta)$ by \cite[Theorem~4.8]{lrt}.

It remains to show that $(f_{\alpha,\beta})_{(\alpha,\beta)}$ converges weakly to $f$. Let $\varepsilon>0$ and $g\in C(K,X)^*$. Since
\[
C(K,X)^*=\ell_1(K)\pten X^*=\ell_1(K,X^*)
\]
(see \cite[Theorem 14.24]{RojoYAmarillo}), let $(t_n)_n\subseteq K$ be the (countable) support of $g$.
Choose $n_0\in\mathbb{N}$ such that
\[
\sum_{n\geq n_0}\|g(t_n)\|<\tfrac{\varepsilon}{4}.
\]
Then there exists $\beta_0$ such that $U_{\beta_0}$ separates $t_0$ from $\{t_1,\ldots,t_{n_0-1}\}\setminus\{t_0\}$. Hence, for all $\beta\geq\beta_0$,
\[
|\langle g,f_{\alpha,\beta}-f\rangle|
\leq |\langle g(t_0),x_\alpha-f(t_0)\rangle|
     +\sum_{n\geq n_0}\|g(t_n)\|(\|f_{\alpha,\beta}\|+\|f\|).
\]
Choosing $\alpha_0$ large enough, we obtain
\[
|\langle g,f_{\alpha,\beta}-f\rangle|<\varepsilon,
\]
which proves the weak convergence.

(c). It is Proposition \ref{prop:C0(L,X)} observing that $C(K,X)=C_0(K,X)$.
\end{proof}

We conclude the paper by showing that, in a certain sense, the weak-DD2P behaves differently from other diameter two properties.

\begin{remark}\label{rem:C(K,X)vecval}
Observe that, by virtue of assertion~(b) in the above theorem, one can construct examples of spaces $C(K,X)$ with the weak-DD2P for which both $C(K)$ and $X$ fail the weak-DD2P. For instance, the space $C([0,\omega_0],\ell_2)$ has the weak-DD2P, even though $C([0,\omega_0])$ fails the weak-DD2P since $[0,\omega_0]$ has a unique cluster point, by Proposition~\ref{prop:denseC(K)}, and $\ell_2$ fails every diameter two property due to its reflexivity. 

To the best of our knowledge, the weak-DD2P is the first known property within the diameter two framework that may hold in a space of the form $C(K,X)$ even when both $C(K)$ and $X$ fail to enjoy this property.
\end{remark}

\section*{Acknowledgements}  

This work was initiated during the first-named author’s visit to Universidad de Granada in March 2026 supported by ``Maria de Maeztu'' Excellence Unit IMAG (CEX2020-001105-M), funded by MICIU/AEI/10.13039/501100011033  and by Junta de Andaluc\'ia. He gratefully acknowledges the support received during his stay and warmly thanks all those who contributed to making his visit a very pleasant and rewarding experience.

The research of J.\ Guerrero-Viu was supported by FPU24/02284 predoctoral grant funded by MCIU and by grant PID2022-137294NB-I00 funded by MCIN/AEI/10.13039/501100011033 and by “ERDF A way of making Europe”. The research of M.\ Mart\'in and A.\ Rueda Zoca has been supported by MICIU/AEI/10.13039/501100011033 and ERDF/EU through the grant PID2021-122126NB-C31, by ``Maria de Maeztu'' Excellence Unit IMAG (CEX2020-001105-M) funded by MICIU/AEI/10.13039/501100011033 and by Junta de Andaluc\'ia, and by Junta de Andalucía grant FQM-185.

The authors acknowledge the use of artificial intelligence–based tools, including Microsoft Copilot 365, ChatGPT (OpenAI), and Gemini (Google), for language editing and stylistic improvements of the manuscript. The authors remain fully responsible for the content of this work.

\end{document}